\documentclass[a4paper,10pt]{article}
\usepackage[T1]{fontenc}
\usepackage[cp1250]{inputenc}
\usepackage[english]{babel}
\usepackage{amsmath}
\usepackage{amsfonts}
\usepackage{amssymb}
\usepackage{mathrsfs}
\usepackage{amsthm}
\usepackage{graphicx}

\usepackage{color}
\usepackage{euscript}
\usepackage{cite}
\usepackage{mathtools}
\usepackage{enumerate}

\DeclareMathOperator{\supp}{supp}

\DeclareMathOperator{\dist}{dist}
\DeclareMathOperator{\BV}{\mathrm{BV}}

\newcommand{\eps}{\varepsilon}

\DeclareMathOperator{\E}{\mathbb{E}}

\DeclareMathOperator{\J}{J}
\newcommand{\W}{\mathbb{W}}

\renewcommand{\ge}{\geqslant}
\renewcommand{\leq}{\leqslant}
\renewcommand{\geq}{\geqslant}

\newcommand{\F}{\mathcal{F}}
\newcommand{\T}{\mathbb{T}}
\newcommand{\TT}{\mathbb{I}}
\newcommand{\AF}{\mathcal{AF}}

\newcommand{\R}{\mathbb{R}}

\newcommand{\M}{\mathbb{M}}

\newcommand{\scalprod}[2]{\langle{#1},{#2}\rangle}
\newcommand{\eq}[1]{\begin{equation}{#1}\end{equation}}
\newcommand{\mlt}[1]{\begin{multline}{#1}\end{multline}}
\newcommand{\alg}[1]{\begin{align}{#1}\end{align}}

\newcommand{\set}[2]{\{{#1}\mid{#2}\}}
\newcommand{\Set}[2]{\Big\{{#1}\,\Big|\;{#2}\Big\}}

\newcommand{\twodots}{\,..\,}

\newcommand{\Lref}[1]{\stackrel{#1}{\leq}}

\newcommand{\Geqref}[1]{\stackrel{\scriptscriptstyle{\eqref{#1}}}{\geq}}

\newcommand{\avsum}{\mathop{\mathpalette\avsuminner\relax}\displaylimits}

\makeatletter
\newcommand\avsuminner[2]{%
  {\sbox0{$\m@th#1\sum$}%
   \vphantom{\usebox0}%
   \ooalign{%
     \hidewidth
     \smash{\vrule height\dimexpr\ht0+1pt\relax depth\dimexpr\dp0+1pt\relax}%
     \hidewidth\cr
     $\m@th#1\sum$\cr
   }%
  }%
}
\makeatother

\newcommand{\Bell}{\mathbb{B}}
\DeclareMathOperator{\NV}{NV}
\newcommand{\X}{\mathfrak{X}}
\newcommand{\TTT}{\mathfrak{T}}
\newcommand{\Z}{\mathfrak{Z}}
\newcommand{\ZZ}{\mathbb{Z}}
\newcommand{\SSS}{\mathfrak{S}}
\newcommand{\WW}{\mathfrak{W}}
\newcommand{\C}{\mathbb{C}}

\newtheorem{Le}{Lemma}[section]
\newtheorem{Def}{Definition}[section]
\newtheorem{St}[Le]{Proposition}
\newtheorem{Th}{Theorem}[section]
\newtheorem{Cor}[Le]{Corollary}
\newtheorem{Rem}[Le]{Remark}
\newtheorem{Conj}{Conjecture}
\numberwithin{equation}{section}

\newcommand{\Wg}{W_{\nabla}}
\newcommand{\Wd}{W_{\mathrm{div}}}
\newcommand{\Og}{\Omega_{\nabla}}
\newcommand{\Od}{\Omega_{\mathrm{div}}}

\begin{document}

\title{Trace inequalities for Sobolev martingales}

\author{Dmitriy Stolyarov\thanks{Supported by the Russian Science Foundation grant no 19-71-30002.}}

\maketitle


\begin{abstract}
We study limiting trace inequalities in the style of Maz'ya and Meyers--Ziemer for Sobolev martingales. We develop the Bellman function approach to such estimates, which allows to provide sufficient and almost necessary conditions on the martingale space and the martingale transform under which the trace inequalities hold true. 
\end{abstract}

\section{Introduction}
The space of summable functions and its subspaces play a special role in analysis. Many statements that are true for~$L_p$ spaces in the reflexive regime~$1 < p < \infty$ break at the endpoints; one may recall the~$L_p$ boundedness of singular integrals as an example. Other remain true when interpreted properly, but the proofs might be more complicated. This is the instance for Sobolev embedding theorems. While the foundations of this theory in the reflexive regime were laid by Sobolev in~\cite{Sobolev1938}, the endpoint case~$p=1$ was not covered until the work of Gagliardo~\cite{Gagliardo1959} and Nirenberg~\cite{Nirenberg1959}. Even then, many natural interesting questions were left open in the case~$p=1$. One of the directions of research was initiated by Bourgain and Brezis in~\cite{BourgainBrezis2004}. Roughly speaking, we may describe Bourgain--Brezis inequalities as 
\eq{\label{BBGeneral}
\|f\|_{L_p(\R^d)} \lesssim \|A f\|_{L_1(\R^d)},\qquad 
}
where~$A$ is a differential operator of order~$m$ or a Fourier multiplier with similar homogeneity properties,~$p = d/(d-m)$, and~$f$ may satisfy additional constraints (such as~$\mathrm{div} f =0$). Usually, these inequalities are related to vectorial behavior of the function~$f$ or the operators in question ($Af$ is usually a vector field or a differential form). We refer the reader to the papers~\cite{BourgainBrezis2002},~\cite{BourgainBrezis2007}, \cite{BousquetVanSchaftingen2014}, \cite{CVSYu2017},  \cite{HernandezSpector2020}, \cite{KMS2015}, \cite{LanzaniStein2005}, \cite{Mazya2007}, \cite{Mazya2010}, \cite{Raita2018}, \cite{Raita2019}, \cite{VanSchaftingen2004}, \cite{VanSchaftingen2004one}, \cite{VanSchaftingen2008}, \cite{VanSchaftingen2010}, \cite{VanSchaftingen2013}, \cite{SpectorVanSchaftingen2019}, \cite{Stolyarov2021},  \cite{Stolyarov2020},  among many others, as well as the surveys~\cite{VanSchaftingen2014} and~\cite{Spector2020} for many Bourgain--Brezis inequalities.

There is a classical trace inequality
\eq{\label{MMZ}
\|f\|_{L_1(\mu)} \lesssim  \Big(\sup\limits_{\genfrac{}{}{0pt}{-2}{x\in \R^d}{r > 0}} r^{1-d}\mu(B_r(x))\Big)\|\nabla f\|_{L_1(\R^d)},\qquad f\in C_0^\infty(\R^d),
}
proved by Maz'ya in~\cite{Mazya1975} and Meyers and Ziemer in~\cite{MeyersZiemer1977} (see~\cite{Mazya2011} for generalizations), which, seemingly does not have any relatives among Bourgain--Brezis inequalities. We call this inequality the trace inequality, because it allows to define traces of~$W_1^1$ and even~$\BV$ functions on sets of non-zero~$\mathcal{H}^{d-1}$-measure (we plug for~$\mu$ the natural Frostman measure related to the set). It is desirable to obtain similar inequalities in the generality of~\eqref{BBGeneral}; first, because it is an~$L_1\to L_1$ inequality, which seems to be more difficult than the usual~$L_1\to L_p$ Bourgain--Brezis inequalities; second, such type inequalities often deliver information about geometric properties of corresponding~$\BV$-type measures. The depth of the desirable analog of~\eqref{MMZ} is emphasized by the fact that defining the trace on a hyperplane, the simplest~$(d-1)$-dimensional set, in the generality of~\eqref{BBGeneral} was already a demanding problem. We refer the reader to~\cite{BDG2020},~\cite{DieningGmeineder2021},~\cite{GRV2022}, and~\cite{GRvS2019} for trace inequalities on hyperplanes for differential operators and~$L_1$-norms.

In this paper, we do not aim to say anything new about functions on~$\R^d$. We wish to explore the aforementioned class of inequalities in the related martingale setting introduced in~\cite{ASW2018} (the discrete model might be traced to~\cite{Janson1977}). While there is a well-known analogy between singular integrals on~$\R^d$ and martingale transforms, the existence of an effect similar to that of Bourgain--Brezis inequalities in the martingale setting was noticed only in~\cite{ASW2018}. The consideration of these related discrete problems allows to guess the right way to prove statements in the Euclidean case. See~\cite{Stolyarov2020} for the proof of several endpoint Bourgain--Brezis inequalities with the approach suggested by~\cite{ASW2018}, \cite{Stolyarov2019} for a probabilistic model of weakly cancelling operators and the implementation of the reasoning of that paper on the Euclidean setting in~\cite{Stolyarov2021}. See~\cite{Stolyarov2021bis} for Maz'ya's $\Phi$-inequalities in the martingale setting and~\cite{Stolyarov2021bibis} for the transference of that reasoning to the Euclidean setting.

In the forthcoming section, we introduce the martingale model. Section~\ref{S2} contains description of our methods. In~\cite{ASW2018}, the authors relied upon common combinatorial tricks for proving martingale inequalities (mostly stopping times), here we will need to use a finer tool called the Bellman function or the Burkholder method. We refer the reader to the foundational papers~\cite{Burkholder1984},~\cite{NazarovTreil1996} and the books~\cite{Osekowski2012}, \cite{VasyuninVolberg2020} for the basics of the Bellman function method. Though the present paper is self-contained, it might be instructive to consult~\cite{Stolyarov2021bis}, where a simpler Bellman function was used to solve a related problem. For smoothness of exposition, we first prove the non-limiting~$L_1$-trace inequality from~\cite{ASW2018} using Bellman function in Section~\ref{S3}. This is already interesting, because the Bellman extremal problem arising here seems new. Section~\ref{S4} introduces additional requirements on the spaces and functionals and states the main result, Theorem~\ref{Main}. The assumptions we impose on the martingale Sobolev space  (see Definition~\ref{GeometricSpace}) might be thought of as martingale versions of the requirements that the~$k$-wave cone is empty introduced in~\cite{APHF2019} (see~\cite{Stolyarov2020bis} as well). This is a simple condition on the operator~$A$ as in~\eqref{BBGeneral} sufficient to conclude that any vectorial measure~$\mu =Af$ has the lower Hausdorff dimension at least~$k$; note that in some cases this condition might be essentially sharpened, see~\cite{Ayoush2021}. Section~\ref{S5} contains the proof of the main theorem, which has been reduced to a verification of a finite dimensional inequality by the Bellman function method in Section~\ref{S3}; the said finite dimensional inequalities are still non-trivial. The final section is devoted to two examples that emphasize the analogy between the martingale problems and similar questions in the Euclidean setting. We consider two particular cases of martingale Sobolev spaces that resemble the spaces~$\dot{W}_1^1$ and the space of divergence-free measures and apply our investigations to these two cases; the results are summarized in Theorems~\ref{GraadTh} and~\ref{DivTh}. We end the paper with formulating a conjecture about functions on the Euclidean space.

\section{Sobolev martingales}\label{S1}
Consider a probability space. If~$S$ is a finite algebra of measurable sets, a non-empty set~$a\in S$ is called an atom provided there are no non-empty sets~$b\in S$ such that~$b\subset a$. Suppose that~$m \geq 2$ is a fixed natural number and~$\F = \{\F_n\}_{n \geq 0}$ is an~$m$-uniform filtration (i.e. an increasing sequence of set algebras,~$\F_n\subset \F_{n+1}$, such that each atom in~$\F_n$ is split into~$m$ atoms in~$\F_{n+1}$ having equal masses; we also assume that~$\F_0$ is the trivial algebra)  on the standard atomless probability space. The set of atoms of the algebra~$\F_n$ is denoted by~$\AF_n$. There is a natural tree structure on the set of all atoms~$\cup_n \AF_n$: the atom~$\omega' \in \AF_{n+1}$ is a kid of~$\omega \in \AF_{n}$ if~$\omega' \subset \omega$. In such a case,~$\omega$ is the parent of~$\omega'$, and we denote the parent of~$\omega'$ by~$(\omega')^{\uparrow}$. For each atom~$a\in \AF_n$, we enumerate its kids with the numbers~$1,2,\ldots,m$ and fix this enumeration.

We may treat our probability space equipped with an~$m$-uniform filtration as a metric space~$\mathbb{T}$. The points of~$\mathbb{T}$ are infinite paths in the tree of atoms (we start from the whole space, then choose an atom in~$\AF_1$, then pass to one of its kids in~$\AF_2$, and so on). More formally, a path is a mapping~$\gamma \colon \mathbb{N}\cup\{0\} \to [1\twodots m]$. A point of~$\mathbb{T}$ that corresponds to such a mapping~$\gamma$ may be thought of as the intersection of the atoms~$\omega_n\in \AF_n$,~$n=0,1,2,\ldots$, where~$\omega_{n+1}$ is the~$\gamma(n)$-th kid of~$\omega_n$. This allows to interpret atoms as subsets of~$\mathbb{T}$. The distance between the two paths~$\gamma_1$ and~$\gamma_2$ is defined by the formula
\begin{equation*}
\dist(\gamma_1,\gamma_2) = m^{-d}, \quad d = \max\{n\in\mathbb{N}\mid \gamma_1(j) = \gamma_2(j) \hbox{ for all } j < n\}.
\end{equation*}  
With this metric,~$\mathbb{T}$ becomes a compact metric space
.   Another way to interpret~$\mathbb{T}$ is to identify a path~$\gamma = (\gamma(0),\gamma(1),\gamma(2),\ldots)$ with the formal series~$\sum_{n=0}^\infty(\gamma(n)-1)m^{-n-1}$, 
which induces the surjection~$\T\to [0,1]$. This interpretation suggests to call~$\gamma(n)$ the~$n$-th digit of~$\gamma$. In this interpretation, the atoms become~$m$-adic subintervals of~$[0,1]$. One may prove that~$\mathbb{T}$ is homeomorphic to a Cantor-type set.

Recall that a sequence of random variables~$F = \{F_n\}_n$ is called a martingale adapted to~$\F$, provided, first, each~$F_n$ is~$\F_n$-measurable, and second,~$F_{n} = \E(F_{n+1}\mid \F_n)$ for any~$n\ge 0$. For a martingale~$F$ adapted to~$\F$, let~$\{dF_n\}_{n \geq 1}$ be the sequence of its martingale differences:
\begin{equation*}
dF_{n+1} = F_{n+1} - F_n,\qquad n \geq 0.
\end{equation*} 
We refer the reader to Chapter VII in~\cite{Shiryaev2019} for the general martingale theory. We will be working with a very narrow class of martingales adapted to uniform filtrations and will not use anything other than terminlogy from the general theory.

We will be using the notation~$[1\,..\, m]$ to denote the integer interval with the endpoints~$1$ and~$m$, including the endpoints. Consider the linear space
\begin{equation*}
V = \Set{v\in\mathbb{R}^m}{\sum_{1}^m v_j = 0}.
\end{equation*}
For each atom~$\omega \in \AF_n$, the function~$dF_{n+1}$ attains at most~$m$ values on~$\omega$, and, thus, might be identified with an element of~$V$.  Namely, for any~$\omega$, we fix a bijective map
\begin{equation*}
\J_{\omega}\colon [1\,..\,m]\to \{\omega' \in \AF_{n+1}\mid \omega = (\omega')^\uparrow\}
\end{equation*}
that maps the number~$j$ to the~$j$-th kid of~$w$. This map may be extended linearly to the mapping between~$V$ and the space of restrictions of all possible martingale differences~$dF_{n+1}$ to~$\omega$. This extended map is also called~$\J_{\omega}$.

Let us observe that each signed Borel measure~$\mu$ with bounded variation on~$\T$ generates a martingale via the formula
\begin{equation}\label{MeasureGeneratesAMartingale}
F_n = m^n\sum\limits_{\omega\in \AF_n} \mu(\omega)\chi_{\omega},\qquad n \in \mathbb{N} \cup \{0\}.
\end{equation}
Note that the characteristic functions of atoms are continuous with respect to the metric~$d$ and form a total family in the space~$C(\mathbb{T})$ (i.e. each measure on~$\mathbb{T}$ is uniquely defined by its values at the atoms). Using this fact one may establish the one-to-one correspondence via~\eqref{MeasureGeneratesAMartingale} between finite signed measures on~$\mathbb{T}$ and~$L_1$-martingales adapted to~$\F$. By an~$L_1$ martingale we mean a martingale~$F$ such that~$\sup_n\|F_n\|_{L_1}$ is finite, the latter expression is the~$L_1$ norm of~$F$.

The Hardy--Littlewood--Sobolev inequality for martingales (going back to~\cite{Watari1964}) reads as follows: 
\eq{\label{HLSmart}
\Big\|\sum\limits_{n\geq 0}m^{-\alpha n} dF_n\Big\|_{L_q} \lesssim \|F\|_{L_p},
}
where~$1/p - 1/q = \alpha$ and~$1 < p < q < \infty$. Here and in what follows, the notation~$A\lesssim B$ means~$A\leq C B$, where the constant~$C$ is uniform in a certain sense. For example, the constant in~\eqref{HLSmart} should not depend on the particular choice of~$F$, however, it might depend on~$p$ and~$q$. The parameter~$\alpha$ should be interpreted as the order of the integration operator. See~\cite{NakaiSadasue2012} and~\cite{StolyarovYarcev2020} for more general versions of the Hardy--Littlewood--Sobolev inequality for martingales.
\begin{Rem}\label{SubcriticalRemark}
Note that the inequality~\eqref{HLSmart} is true when~$p=1$ and~$1-1/q < \alpha$ since
\eq{
\|dF_n\|_{L_q} \leq m^{\frac{(q-1)(n+1)}{q}}\|dF_n\|_{L_1}\leq 2m^{\frac{(q-1)(n+1)}{q}}\|F\|_{L_1}.
}
\end{Rem}

If the martingale~$F$ is~$\mathbb{R}^{\ell}$-valued, then~$f_{n+1}|_{\omega}$ might be naturally identified with an element of~$V\otimes \R^\ell$ (we apply~$\J_{\omega}$ to each of~$\ell$ coordinates individually), here~$\omega \in\AF_n$. In other words, we may think of the trace of~$f_{n+1}$ on~$\omega \in \AF_n$ as of an~$m\times \ell$ matrix such that the sum of the elements in each row equals zero. Let~$W$ be a linear subspace of~$V\otimes \R^\ell$. Consider the subspace~$\W$ of the~$\mathbb{R}^\ell$-valued martingale space~$L_1$:
\begin{equation*}
\W = \Big\{F \text{ is an } L_1 \text{ martingale with values in } \R^\ell\,\Big|\; \forall n,\ \omega \in \AF_n \quad dF_{n+1}|_\omega \in \J_{\omega}[W]\Big\}.
\end{equation*}
In view of~\eqref{MeasureGeneratesAMartingale}, one should think of~$\W$  as of a~$\BV$-type space. These spaces were implicitly introduced in~\cite{Janson1977} and proved to be good discrete models for several problems in harmonic analysis on~$L_1$ and the space of measures, see~\cite{ASW2018} and \cite{Stolyarov2019}; the former paper suggested the name 'Sobolev martingales' since the behavior of the spaces~$\W$ resembles that of~$\dot{W}_1^1$.

\section{Setting of the trace problem and the Bellman function}\label{S2}
Let~$\varphi\colon W\to V$ be a linear operator 
and let~$\alpha \in (0,1)$. Consider the operator
\begin{equation}\label{Raita's}
\TT_\alpha [F] = \sum\limits_{n} m^{-\alpha n}\sum\limits_{\omega \in \AF_n}\J_\omega\Big[\varphi(\J^{-1}_\omega [dF_{n+1}|_{\omega}])\Big],\qquad F \in \W. 
\end{equation} 
This operator maps martingales to functions; by the correspondence~\eqref{MeasureGeneratesAMartingale}, we may interpret it as an operator that transforms measures into functions (the series converges in~$L_q$ with~$q < 1/(1-\alpha)$ by Remark~\ref{SubcriticalRemark}). The operator~$\J_w$ is used to make the formula mathematically rigorous, informally,~\eqref{Raita's} may be written as
\eq{
\TT_\alpha [F] = \sum\limits_{n} m^{-\alpha n}\sum\limits_{\omega \in \AF_n}\varphi(dF_{n+1}|_{\omega}),\qquad F \in \W.
}
We note that the operator~$\TT_\alpha$ cannot be applied to an arbitrary~$\R^\ell$-valued martingale since~$\varphi$ is defined on~$W$ only and the condition~$dF_{n+1}|_\omega \in W$ is required.
We will also use the notation
\eq{
\Big(\TT_\alpha [F]\Big)_N = \sum\limits_{n=0}^{N-1} m^{-\alpha n}\sum\limits_{\omega \in \AF_n}\J_\omega\Big[\varphi(\J^{-1}_\omega [dF_{n+1}|_{\omega}])\Big],\qquad F \in \W.
}
Note that~$(\TT_\alpha [F])_N$ is~$\F_N$-measurable. Therefore,~$\{(\TT_\alpha [F])_n\}_n$ is a martingale.
\begin{Rem}
Set~$\ell = 1$. The choice of the identity operator for~$\varphi$ reduces~$\TT_\alpha$ to the martingale Riesz potential present in the Hardy--Littlewood--Sobolev inequality~\eqref{HLSmart}. A different choice of the operator~$\varphi$ might make the operator~$\TT_\alpha$ more cancelling than the Riesz potential \textup(usually\textup, larger~$\ell$ are needed for such a choice\textup). This is similar to the difference between the classical Euclidean Riesz potential~$|x|^{\beta}$ and the kernel~$x/|x|^{\beta+1}$. As we will see\textup, these additional cancellations become important only at endpoints \textup(a similar effect appeared in~\textup{\cite{RSS2021}}\textup).
\end{Rem}
We will be studying the question: 'for which measures~$\nu$ on~$\T$ does the inequality
\eq{\label{Trace}
\|\TT_{\alpha}[F]\|_{L_1(\nu)}\lesssim \|F\|_{L_1},\qquad F\in \W,
}
hold true?' This inequality allows to define~$L_1$-traces of~$\TT_\alpha[F]$ on lower-dimensional subsets of~$\T$ (this explains the name 'trace inequlaity'). We will be working with simple martingales only since usually the result may be extended to the class of all~$L_1$-martingales by a routine limiting argument. 
\begin{Def}
	We say that a measure~$\nu$ with bounded variation on~$\mathbb{T}$ satisfies the~$(\alpha,p)$ Frostman condition provided
	\begin{equation}\label{Frostman}
	\forall n \geq 0\quad \forall \omega \in \AF_n \quad (\nu(\omega))^{\frac{1}{p}} \lesssim m^{(\alpha - 1)n}.
	\end{equation}
\end{Def}
One may restate the~$(\alpha,p)$ Frostman conditions in terms of Morrey space norms, see~\cite{RSS2021}. Sometimes, the~$(\alpha,p)$ Frostman conditions is called the ball growth condition.
\begin{Rem}
One may prove that the~$(\alpha,1)$-Frostman condition is necessary for~\eqref{Trace}\textup, provided~$\varphi$ is not degenerate. Namely\textup, assume that for any~$j\in [1\,..\,m]$ there exists~$v\in W$ such that
the~$j$-th coordinate of~$\varphi[v]$ as an element of~$\R^m$\textup, is non-zero \textup(we will later denote this number by~$(\varphi[v])_j$\textup,~$j=1,2,\ldots, m$\textup). Then\textup, the~$(\alpha,1)$-Frostman condition is necessary for~\eqref{Trace}. One may prove this by plugging 'elementary martingales' into~\eqref{Trace}\textup, i.e. martingales~$F$ for which~$dF_{n+1}|_w$ is non-zero for only single~$n$ and single~$w\in \AF_n$.
\end{Rem}
We will be working with a slightly more demanding 'bilinear' form of~\eqref{Trace}:
\eq{\label{BilinearTrace}
\|\TT_{\alpha}[F]\|_{L_1(\nu)}\lesssim\Big(\sup\limits_{\genfrac{}{}{0pt}{-2}{n\geq 0,}{\omega \in \AF_n}}m^{(1-\alpha)n}\nu(\omega)\Big) \|F\|_{L_1} ,\qquad F\in \W.
}
\begin{Def}\label{Def32}
Define the Bellman function~$\Bell \colon \R^\ell\times\R\times \R_+\times \R_+\times\R_+\to \R$ by the rule
\mlt{\label{Bell}
\Bell(x,y,z,t,s) \\ =\sup\bigg(\Set{\int\limits_\T\big|y+\TT_\alpha[F]\big|\,d\nu}{F_0=x,\E|F_{\infty}| = z, \nu(\T) = t, \sup\limits_{\genfrac{}{}{0pt}{-2}{n\geq 0,}{\omega \in \AF_n}}m^{(1-\alpha)n}\nu(\omega) = s, \ F\in\W}\bigg).
}
The supremum in the formula above is taken over the set of simple martingales~$F$ only\textup; note that in such a case we may assume that~$\nu$ is generated \textup(via~\eqref{MeasureGeneratesAMartingale}\textup) by a simple martingale as well.
\end{Def}
The Bellman domain (i.e. the set of points~$(x,y,z,t,s)$ for which the supremum in~\eqref{Bell} is taken over a non-empty set) is
\eq{\label{Domain}
\Omega=\Set{(x,y,z,t,s)\in \R^\ell\times\R\times \R_+\times \R_+\times\R_+}{|x|\leq z, \ t\leq s}.
}
The Bellman function satisfies the boundary inequality (this follows from substitution of a constant martingale into~\eqref{Bell})
\eq{
\Bell(x,y,|x|,t,s) \geq |y|t.
}
The Bellman function also satisfies the main inequality
\eq{\label{MainIneq}
\Bell(x,y,z,t,s) \geq m^{-\alpha}\sum\limits_{j=1}^m\Bell(x_j,m^\alpha y + (\varphi[\vec{x}])_j,z_j,t_j,s_j),
}
where
\mlt{\label{SplittingRules}
x = \avsum\limits_{j=1}^{m} x_j,\ z = \avsum\limits_{j=1}^{m}z_j,\ t = \sum_{j=1}^m t_j, \\ \vec{x} = (x_1-x,x_2-x,\ldots, x_m-x)\in W,\ \text{and}\ s = \max(t,m^{1-\alpha}\max_{j=1}^m s_j).
}
The symbol~$\avsum_{j=1}^{m}$ means the average, i.e.~$\avsum_{j=1}^{m} a_j = \frac{1}{m}\sum_{j=1}^m a_j$.
\begin{St}\label{Prop32}
The Bellman function indeed satisfies the main inequality.
\end{St}
Proposition~\ref{Prop32} is similar to Lemma~$3.4$ in~\cite{Stolyarov2021bis}, so, we omit its proof. Formally, we will not needed. A reverse statement is more important for us. We need to set up notation before formulating it. 
\begin{Def}
The configuration space~$\WW$ is formed by strings~$(\X,y,\Z,\TTT,\SSS)$\textup, where
\mlt{
\X = \{x_j\}_{j=1}^m\in \R^{\ell m},\quad y\in \R,\quad \Z = \{z_j\}_{j=1}^m\in (\R_+)^m,\\
 \TTT = \{t_j\}_{j=1}^m\in (\R_+)^m,\quad \text{and}\quad \SSS = \{s_j\}_{j=1}^m\in (\R_+)^m
}
satisfy the splitting rules~\eqref{SplittingRules}\textup, i.e.~$\vec{x}$ constructed by the formulas listed in~\eqref{SplittingRules} belongs to~$W$\textup, and the natural domain requirements~$|x_j| \leq z_j$ and~$t_j \leq s_j$ for all~$j = 1,2,\ldots,m$ are also fulfilled.
\end{Def}
\begin{St}\label{SuperSolution}
Let~$G\colon \Omega \to \R$ be a function that obeys the main inequality~\eqref{MainIneq}\textup, i.e. such that
\eq{\label{MIG}
G(x,y,z,t,s) \geq m^{-\alpha}\sum\limits_{j=1}^m G(x_j,m^\alpha y + (\varphi[\vec{x}])_j,z_j,t_j,s_j)
} 
for all~$(\X,y,\Z,\TTT,\SSS)\in \WW$ and~$x,z,t,s$ generated from them by~\eqref{SplittingRules}. Let~$\nu_n$ be the martingale generated by~$\nu$ via~\eqref{MeasureGeneratesAMartingale}\textup, and let
\eq{
M_n(\omega) = \sup\limits_{\genfrac{}{}{0pt}{-2}{N \geq n, v\subset \omega,}{v \in \AF_N}} m^{(1-\alpha)N}\nu(v).
}
Then\textup, the process
\eq{\label{Process}
m^{(1-\alpha)n} G(F_n, m^{\alpha n}(y_0+ (\TT_\alpha [F])_n), \E(|F_\infty|\mid \F_n),m^{-n}\nu_n, M_n)
}
is a supermartingale provided~$F\in \W$ and~$y_0\in\R$.
\end{St}
A sequence~$\{H_n\}_{n}$ of random variables is called a supermartingale provided each~$H_n$ is~$\F_n$-measurable (i.e.~$\{H_n\}_n$ is adapted to~$\F$) and for any~$n > 0$
\eq{
\E(H_{n+1}\mid \F_n)\leq H_n.
}
If~$\{H_n\}_n$ is a supermartingale, then 
\eq{\label{SupermartingaleProperty}
\E H_N \leq H_0
}
for any~$N$.

Proposition~\ref{SuperSolution} and Corollary~\ref{SuperSolutionCor} below are similar to Lemma~$3.9$ in~\cite{Stolyarov2021bis}.
\begin{proof}[Proof of Proposition~\ref{SuperSolution}]
Let~$w\in \AF_n$ and let~$w_1,w_2,\ldots, w_m$ be its kids in~$\AF_{n+1}$. Let
\mlt{
y = m^{\alpha n}\Big(y_0 + (\TT_\alpha[F])_{n}(w)\Big);\quad x_j = F_{n+1}(w_j),\ z_j = \E(|F_\infty|\mid \F_{n+1})(w_j),\\
t_j= \nu(w_j)=m^{-n-1}\nu_{n+1}(w_j),\ s_j = M_{n+1}(w_j),\qquad j \in [1\,..\,m].
}
By the martingale properties,
\eq{
x = F_n(w),\ z = \E(|F_\infty|\mid F_n)(w),\ t=\nu(w),\ s=M_n(w),\quad \text{and} \ \vec{x} = \J_{w}^{-1}[dF_{n+1}|_w],
}
satisfy the splitting rules~\eqref{SplittingRules}. In other words, the string~$(\X,y,\Z,\TTT,\SSS)$ formed by these numbers generated as above, belongs to the configuration space~$\WW$. Therefore, the main inequality~\eqref{MIG} applies, which means exactly the supermartingale property
\mlt{
m^{(1-\alpha)n} G\Big(F_n(w), m^{\alpha n}(y_0+ (\TT_\alpha [F])_n(w)), \E(|F_\infty|\mid \F_n)(w),m^{-n}\nu_n(w), M_n(w)\Big)\\ \geq m^{(1-\alpha)(n+1)}\E G\Big(F_{n+1}, m^{\alpha (n+1)}(y_0+ (\TT_\alpha [F])_{n+1}), \E(|F_\infty|\mid \F_{n+1}),m^{-n-1}\nu_{n+1}, M_{n+1}\Big)\chi_w
} 
for the process~\eqref{Process}; we have used the identity
\eq{
m^{\alpha (n+1)}(y_0+ (\TT_\alpha [F])_{n+1}(w_j)) = m^\alpha y + (\varphi[\vec{x}])_j \qquad \text{on}\ w_j.
}

\end{proof}
\begin{Cor}\label{SuperSolutionCor}
Suppose there exists a finite function~$G\colon \Omega \to \R$ that satisfies the main inequality~\eqref{MIG}\textup, the boundary inequality~$G(|x|,y,|x|,s,t) \geq |y|t$\textup, and the estimate
\eq{\label{EstFromAbove}
G(x,0,z,t,s) \lesssim zs.
}
Then\textup, the inequality~\eqref{BilinearTrace} holds true for all measures~$\nu$ that satisfy the~$(\alpha,1)$-Frostman condition.
\end{Cor}
\begin{proof}
By Definition~\ref{Def32}, it suffices to prove the estimate
\eq{\label{BellmanInequality}
\Bell(x,0,z,t,s)\lesssim zs.
}
In fact, we will prove that~$\Bell \leq G$ on~$\Omega$. Let~$(x,y,z,t,s)\in\Omega$ and let~$(F,\nu)$ be a pair of a simple martingale and a measure generated by a simple martingale that fulfills the requirements in~\eqref{Bell} for these fixed~$x,z,t,s$. By Proposition~\ref{SuperSolution}, the corresponding process~\eqref{Process} is a supermartingale. Therefore,
\mlt{
G(x,y,z,t,s) =  G(F_0,y+\TT_\alpha [F]_0, \E|F_\infty|,\nu_0, M_0)\\ \Geqref{SupermartingaleProperty} m^{(1-\alpha)N}\E  G\Big(F_N, m^{\alpha N}(y+ (\TT_\alpha [F])_N), \E(|F_\infty|\mid \F_N),m^{-N}\nu_N, M_N\Big),
}
for any~$N > 0$. Let us pick sufficiently large~$N$ (recall we are working with simple martingales). By the boundary inequality, the latter expression is bounded from below by
\eq{
m^{(1-\alpha)N}\E m^{\alpha N}|y +(\TT_\alpha[F])_N|m^{-N}\nu_N = \E |y +(\TT_\alpha[F])_N|\nu_N = \int\limits_{\T}|y +\TT_\alpha[F]|\,d\nu.
}
Since~$F$ and~$\nu$ are arbitrary, we get
\eq{
\Bell(x,y,z,t,s) \leq G(x,y,z,t,s),
}
which yields~\eqref{BellmanInequality}.
\end{proof}
\begin{Rem}
It is unclear whether the estimate~\eqref{EstFromAbove} in the above corollary is necessary. Maybe\textup, it is possible to derive~\eqref{EstFromAbove} from the finiteness of~$G$ and the validity of the main inequality for it.
\end{Rem}
\begin{Def}
The functions~$G$ that satisfy the main inequality and the boundary inequality are usually called \emph{supersolutions}.
\end{Def}

\section{Subcritical case}\label{S3}
Let~$v\in V$ be a vector. Consider the function~$\kappa_v\colon(0,1]\to \mathbb{R}$ given by the formula
\begin{equation*}
\kappa_v(\theta) = \theta\log\Big(\avsum\limits_{j=1}^{m}|1+v_j|^{\frac{1}{\theta}}\Big) = \log\|{\bf 1}+ v\|_{L_{\frac{1}{\theta}}}.
\end{equation*}
We may extend this function to~$0$ by continuity. Then,
\eq{\label{Kappa0}
\kappa_v(0) = \log \Big(\max\limits_{j=1}^m|1+v_j|\Big).
}
By H\"older's inequality, this function is convex and non-increasing. If, in addition,~$v_j \geq -1$ for any~$j$, this function also satisfies~$\kappa_v(1) = 0$. Consider yet another function~$\kappa\colon [0,1]\to \R$:
\begin{equation}\label{KappaDefinition}
\kappa(\theta) = \sup\Big\{\kappa_v(\theta)\;\Big|\, \exists a\in\mathbb{R}^\ell\setminus\{0\} \hbox{ such that } v\otimes a\in W \hbox{ and } \forall j \quad v_j \geq -1\Big\}.
\end{equation}
The function~$\kappa$ is also convex, non-increasing, and satisfies~$\kappa(1) = 0$.
As an elementary computation shows,
\begin{equation}\label{FormulaForDerivativeOfKappa}
\kappa'(1) = \inf\Big\{-\avsum\limits_{j=1}^{m}(1+v_j)\log(1+v_j)\;\Big|\, \exists a\in\mathbb{R}^\ell\setminus \{0\} \hbox{ such that } v\otimes a\in W \hbox{ and }
 \forall j \quad v_j \geq -1\Big\}.
\end{equation}
We cite Theorem~$6.2$ from~\cite{ASW2018} (there was no operator~$\varphi$ in~\cite{ASW2018}; the theorem below is not sensitive to~$\varphi$).
\begin{Th}\label{TraceTheorem}
	Assume~$\nu$ satisfies the~$(\alpha,1)$ Frostman condition and
	\begin{equation}\label{LowerBoundForAlpha}
	\alpha > -\frac{\kappa'(1)}{\log m}.
	\end{equation}
	Then\textup,~\eqref{BilinearTrace} holds true with any operator~$\varphi$.
\end{Th}
Let us give the Bellman function proof of this theorem (it relies upon the same principles as the original one in~\cite{ASW2018}, but uses different language).
\begin{Th}\label{SubcriticalTheorem}
Assume~\eqref{LowerBoundForAlpha}. Then\textup, for any~$\theta \in (0,1)$ sufficiently close to one, the function
\eq{\label{SubcriticalSupersolution}
G(x,y,z,t,s) = |y|t+ M_1|x|t^{1-\theta}s^{\theta} + M_2 (z-|x|)s
}
is a supersolution\textup, provided~$1\ll M_1 \ll M_2$ (depending on~$\theta$).
\end{Th}
The notation~$\ll$ means the statement holds true when~$M_1$ is sufficiently large and~$M_2$ is also sufficiently large (depending on the particular choice of~$M_1$). Note that Theorem~\ref{SubcriticalTheorem} implies Theorem~\ref{TraceTheorem} via~Corollary~\ref{SuperSolutionCor}.
\begin{Def}
If~$G\colon \Omega \to \R$ is a function\textup, then its discrepancy~$\NV[G]\colon \WW \to \R$ is
\eq{
\NV[G](\X,y,\Z,\TTT,\SSS) = G(x,y,z,t,s) - m^{-\alpha}\sum\limits_{j=1}^m G(x_j,m^\alpha y + (\varphi[\vec{x}])_j,z_j,t_j,s_j).
}
The main inequality says that~$\NV[G] \geq 0$ on~$\WW$. 
\end{Def}
\begin{proof}[Proof of Theorem~\ref{SubcriticalTheorem}]
We wish to prove
\eq{
\NV\Big[|y|t+ M_1|x|t^{1-\theta}s^{\theta} + M_2 (z-|x|)s\Big] \geq 0.
}
Let us first compute~$\NV[(z-|x|)s]$:
\mlt{\label{MainDiscr}
\NV[(z-|x|)s] = (z-|x|)s - m^{-\alpha}\sum\limits_{j=1}^m(z_j-|x_j|)s_j \\= \Big(\avsum\limits_{j=1}^{m}z_j - \avsum\limits_{j=1}^{m}|x_j|\Big)s+ \Big(\avsum\limits_{j=1}^{m}|x_j| - |x|\Big)s -m^{-\alpha}\sum\limits_{j=1}^m(z_j-|x_j|)s_j\\=
\Big(\avsum\limits_{j=1}^{m}|x_j| - |x|\Big)s + \sum\limits_{j=1}^m(z_j - |x_j|)(m^{-1}s - m^{-\alpha}s_j).
}
We see that this expression is always non-negative due to the splitting rules~\eqref{SplittingRules}. What is more, 
\eq{\label{ConvexAtom}
\NV[(z-|x|)s] \gtrsim |\X|s,
}
provided
\eq{\label{Flatness}
\avsum\limits_{j=1}^{m}|x_j| \geq (1+\eps)|x|
} 
for some fixed small~$\eps$ (the notation~$|\X|$ means the Euclidean norm of~$\X = (x_1,x_2,\ldots,x_m) \in \R^{\ell m}$). Now we turn to the function~$|x|t^{1-\theta}s^{\theta}$ and estimate its discrepancy. Let~$p$ be such that~$(1-\theta)p'=1$, where~$1/p + 1/p'=1$. Then,
\mlt{\label{Holder}
m^{-\alpha}\sum\limits_{j=1}^m|x_j|t_j^{1-\theta}s_j^\theta = m^{1-\alpha}\avsum\limits_{j=1}^{m}|x_j|t_j^{1-\theta}s_j^\theta \leq m^{1-\alpha}\Big(\avsum\limits_{j=1}^{m}|x_j|^p\Big)^{1/p}\Big(\avsum\limits_{j=1}^{m} t_j^{(1-\theta)p'}s_j^{\theta p'}\Big)^{1/p'}\\ \Lref{\scriptscriptstyle (1-\theta)p'=1}
m^{1-\alpha} \Big(\avsum\limits_{j=1}^{m} |x_j|^p\Big)^{1/p}\Big(\frac{t}{m}(m^{(\alpha-1)} s)^{\theta p'}\Big)^{1/p'} = m^{(1-\alpha)(1-\theta) - (1-\theta)}\Big(\avsum\limits_{j=1}^{m} |x_j|^p\Big)^{1/p} t^{1-\theta}s^{\theta}.
}
Thus, if
\eq{\label{LpIncr}
m^{-\alpha(1-\theta)}\Big(\avsum\limits_{j=1}^{m} |x_j|^p\Big)^{1/p} \leq |x|,
}
then~$\NV[|x|t^{1-\theta}s^{\theta}] \geq 0$, i.e., the discrepancy of the second term in~\eqref{SubcriticalSupersolution} is non-negative. It was proved in Lemma $2.1$ of \cite{ASW2018} that, given~$p\in(1,\infty]$, for any~$\tilde{\delta} > 0$ there exists~$\eps > 0$ such that 
\eq{
e^{-\kappa(p^{-1})} \Big(\avsum\limits_{j=1}^{m} |x_j|^p\Big)^{1/p} \leq (1+\tilde{\delta})|x|,\qquad \vec{x}\in W,
}
provided~\eqref{Flatness} is violated. Therefore, if~$\alpha > -\kappa'(1)/\log m$,~$p$ is sufficiently close to one and~$\eps$ is sufficiently small, then there exists a tiny~$\delta$ such that
\eq{
m^{-\frac{\alpha p}{p-1}}\Big(\avsum\limits_{j=1}^{m} |x_j|^p\Big)^{1/p} \leq (1-\delta)|x|
}
whenever~$\vec{x}\in W$ and~$\avsum_{j=1}^{m}|x_j| \leq (1+\eps)|x|$, which is slightly stronger than~\eqref{LpIncr}.  Therefore,
\eq{\label{FlatAtom}
\NV[|x|t^{1-\theta}s^{\theta}] \gtrsim |\X|t^{1-\theta}s^{\theta},
}
provided~\eqref{Flatness} is violated with sufficiently small~$\eps$. We fix such an~$\eps$.

Combining~\eqref{ConvexAtom},~\eqref{FlatAtom}, and the simple estimate
\eq{\label{SimpleConvex}
\Big|\NV[|x|t^{1-\theta}s^{\theta}]\Big| \lesssim |\X|t^{1-\theta}s^{\theta},
} 
we obtain
\eq{\label{Discr}
\NV[|x|t^{1-\theta}s^{\theta} + M (z-|x|)s] \gtrsim |\X|t^{1-\theta}s^{\theta},
} 
provided~$M$ is sufficiently large (note that~$t^{1-\theta}s^{\theta} \leq s$). Indeed, in the case~\eqref{Flatness} holds true,~\eqref{Discr} follows from~\eqref{ConvexAtom} and~\eqref{SimpleConvex} by the choice of sufficiently large~$M$; in the case~\eqref{Flatness} is violated, we rely upon the positivity of~$\NV[(z-|x|)s]$ and~\eqref{FlatAtom}\footnote{We will refer to this simple reasoning as to the flat/convex argument (see the explanation after the proof of the lemma).}. 

Thus, it remains to prove
\eq{
\big|\NV[|y|t]\big| \lesssim |\X|t.
}
This follows from direct computation
\eq{
\NV[|y|t] = \sum\limits_{j=1}^m\Big(|y| - |y+m^{-\alpha}(\varphi[\vec{x}])_j|\Big)t_j
}
and the triangle inequality.
\end{proof}
\begin{Rem}
Lemma~$6.2$ in~\textup{\cite{ASW2018}} says that if the function~$\kappa$ is affine and~\eqref{Trace} holds true with any~$\varphi$\textup, then~\eqref{LowerBoundForAlpha} is also true.
\end{Rem}
In~\cite{ASW2018}, condition~\eqref{Flatness} played the pivotal role (see Definition~$3.1$ in that paper). The atoms for which the corresponding quantities violated it were called~$\eps$-flat\footnote{The term 'flat' symbolizes that the martingale develops along the flat part of the boundary of the ball of~$L_1$.}. The core of the reasoning is contained in Lemma~$2.1$ of~\cite{ASW2018}, which manifests the principle that flat atoms (or flat configurations~$\X = (x_1,x_2,\ldots,x_m)$ such that~$\vec{x} \in W$) lie close to rank-one configurations.
It is convenient to introduce the set of admissible~$0$-flat increments:
\eq{\label{MW}
\M^W = \Set{v\in V}{\forall j \in [1\,..\,m]\quad v_j \geq -1 \text{ and } \exists a\in \R^\ell \setminus \{0\} \text{ such that } v\otimes a\in W}.
}
See~\cite{Stolyarov2020} for 'translation' of these notions into the language of the Euclidean space (instead of~$\T$). The following lemma justifies the heuristic principle that flat configurations are close to rank-one configurations. 
\begin{Le}\label{FlatRankOne}
Fix~$W$. Let~$U$ be a neighborhood of the set
\eq{\label{SetU}
\Set{\X \in \R^m\otimes \R^\ell}{\exists\ v \in \M^W, a\in \R^\ell \setminus \{0\}\quad \text{such that}\ \vec{x} = v\otimes a, x=a}.
}
For sufficiently small~$\eps$\textup, the violation of~\textup{\eqref{Flatness}} together with the conditions~$|x|=1$ and~$\vec{x} \in W$ implies~$\X\in U$.
\end{Le}
\begin{proof}
Note that if~$\X = \{x_j\}_{j=1}^m$ violates~\eqref{Flatness} and~$|x|=1$, then~$|x_j|\leq m+1$ for every~$j\in[1\twodots m]$ (we assume~$\eps$ is sufficiently small). Therefore, the configurations~$\X$ we consider are uniformly bounded as vectors in~$\R^{m\ell}$. Assume the contrary: for any~$n \in \mathbb{N}$ there exists~$\X^n = (x_1^n,x_2^n\ldots, x_m^n)\in \R^{m\ell}$ that does not belong to~$U$,~$|x^n|=1$,~$\vec{x}^n \in W$, and
\eq{
\avsum\limits_{j=1}^m|x_j^n| \leq(1+1/n).
}
Let~$\X$ be a limit point of the sequence~$\{\X^n\}_n$. Then,~$\avsum_{j=1}^m|x_j| = 1$, which means the triangle inequality~$\avsum_{j=1}^m|x_j| \geq 1$ turns into equality and there exist~$v\in V$ with~$v_j \geq -1$,~$j\in [1\twodots m]$, such that~$\vec{x} = v\otimes x$. Since~$\vec{x} \in W$, we have~$v\in \M^W$. Therefore,~$\vec{x}$ belongs to~\eqref{SetU}, which contradicts the openness of~$U$.
\end{proof}

\section{Additional requirements on~$\varphi$ and~$W$}\label{S4}
\begin{Def}\label{GeometricSpace}
We call the space~$W\subset V\otimes \R^\ell$ geometric if the corresponding function~$\kappa$ is affine.
\end{Def}
The condition that a space~$W$ is geometric is similar to the condition that the~$k$-wave cone is empty (see also Lemma~\ref{FourierLemma} below); the latter condition played an important role in~\cite{APHF2019} and~\cite{Stolyarov2020bis}.
\begin{St}\label{GeometricStructure}
The space~$W$ is geometric if there exists~$\alpha$ such that
\begin{enumerate}[1\textup{)}]
\item  the inequality
\eq{\label{entropyIneq}
\avsum\limits_{j=1}^{m}(1+v_j)\log(1+v_j) \leq \alpha \log m
}
holds true for any vector~$v\in \M^W$\textup;
\item there exists~$v\in \M^W$ and~$H\subset [1\, ..\, m]$ such that~$|H| = m^{1-\alpha}$ and
\eq{
v_j = \begin{cases}
m^{\alpha }-1,\quad &j\in H;\\
-1,\quad &j\notin H,
\end{cases}
}
in particular\textup{,~\eqref{entropyIneq}} turns into equality on this vector.
\end{enumerate}
\end{St}
\begin{proof}
Let~$W$ be a geometric space and let~$\alpha = -\kappa'(1)/\log m$. Then, since~$\kappa$ is affine, we have~$\kappa(0) = \alpha \log m$. Thus, for arbitrary~$v\in \M^W$, we have~$(1+v_j) \leq m^{\alpha}$ for all~$j$, and, therefore,
\eq{\label{eq53}
\avsum\limits_{j=1}^{m}(1+v_j)\log(1+v_j)\leq \avsum\limits_{j=1}^{m}(1+v_j)\alpha \log m = \alpha \log m.
}
Now let~$v$ be the vector at which the infimum in~\eqref{FormulaForDerivativeOfKappa} is attained (by compactness, such a vector exists). The inequality~\eqref{eq53} turns into equality by the choice of~$v$. Thus, for any~$j \in [1\twodots m]$,~$1+v_j$ equals either zero or~$m^{\alpha}$. The set where it equals~$m^{\alpha}$ is the searched-for set~$H$; it is clear that~$|H| = m^{1-\alpha}$ since~$v$ has mean zero.
\end{proof}
\begin{Rem}
If~$W$ is a geometric space\textup, then~$m^{1-\alpha}$ is an integer.
\end{Rem}
\begin{Rem}
Proposition~\textup{\ref{GeometricStructure}} may be reversed\textup, the conditions described there are not only necessary\textup, but also sufficient for~$W$ to be a geometric space.
\end{Rem}

\begin{Def}
The special vectors~$v\otimes a$ from Proposition~\textup{\ref{GeometricStructure}} for which~\textup{\eqref{entropyIneq}} turns into equality are called \emph{extremal}. The corresponding sets~$H$ are their supports\footnote{This is a slight abuse of notation since~$H$ is the support of the function $j \mapsto 1+ v_j$,~$j \in [1\twodots m]$.}.
\end{Def}
\begin{Rem}
There might be several extremal vectors related to~$a \in \R^\ell \setminus \{0\}$.
\end{Rem}
\begin{Rem}\label{ExtremalityRemark}
The proof of Proposition~\textup{\ref{GeometricStructure}} also says\textup{:} if~$W$ is a geometric space and~$\kappa_v(\theta) = \kappa(\theta)$ for some~$v\in \M^W$\textup,~$\theta\in (0,1)$\textup, then~$v\otimes a$ is an extremal vector for some~$a\in \R^\ell \setminus \{0\}$. Indeed\textup, by Jensen's inequality\footnote{We use the weighted Jensen's inequality $\sum \alpha_j f(a_j) \leq f(\sum \alpha_j a_j)$ for a concave function~$f$ and substitute~$\alpha_j := (1+v_j)/m$,~$a_j := (1+v_j)^{1/\theta-1}$, and~$f:= \log$.},
\eq{
\avsum\limits_{j=1}^m (1+v_j)\log(1+v_j) \leq \frac{1}{1-\theta}\log\Big(\avsum\limits_{j=1}^m (1+v_j)^{1/\theta}\Big)^\theta.
}
We know this inequality turns into equality. Thus\textup, for each~$j$ either~$1+v_j = 0$ or~$1+v_j = m^\alpha$.
\end{Rem}

\begin{Def}\label{CancelingOp}
Let~$W\subset V\otimes \R^\ell$ be a geometric space. We say that an operator~$\varphi \colon W\to V$ is \emph{canceling} if the identity
\eq{
\forall j \in H\qquad (\varphi[w])_j = 0
}
holds true for any extremal vector~$w=v\otimes a$ and its support~$H\subset [1\,..\,m]$.
\end{Def}
The following proposition resembles Theorem~$1.5$ in~\cite{RSS2021}.
\begin{St}
Let~$W$ be a geometric space. If~\eqref{BilinearTrace} holds true at the endpoint~$\alpha = -\kappa'(1)/\log m$\textup, then~$\varphi$ is a canceling operator.
\end{St}
\begin{proof}
Let~$v\otimes a$ be an extremal vector. Consider the martingale~$F$ given by
\eq{\label{eq56}
F_n = \prod\limits_{j=0}^n(1+h_j)\otimes a,\qquad \text{where}\quad h_j=\sum\limits_{w\in \AF_j}\J_w[v].
}
In such a case,~$dF_{n+1} = F_n h_{n+1}\otimes a$ and~$F\in \W$. What is more,
\eq{
\TT_\alpha[F] = \sum\limits_{n \geq 0} m^{-\alpha n}\sum\limits_{w\in \AF_n}\J_w[\varphi[v\otimes a]]F_n(w).
}
Let~$H$ be the support of~$v$. By~\eqref{eq56}, we have~$F_n(w) = 0$ provided there is a digit in the $m$-ary expansion of the 'number' $w$ that does not belong to~$H$. Let us compute the values of~$F_n(w)$ at the other 'numbers'~$w\in \AF_n$:
\eq{
F_n(w) = m^{\alpha n} a.
}
Therefore, if we pick some large~$N$ and consider the stopped martingale~$F^N=\{F_{\min(n,N)}\}_n$, then
\eq{
\TT_\alpha [F^N](w)= \sum\limits_{j=0}^N\big(\varphi[v\otimes a]\big)_{w(j)}
}
at any point~$w$ inside the set
\eq{
\mathbb{H}_N = \Set{w\in \AF_n}{\forall j \leq N\quad w(j)\in H}.
}
The following two statements will contradict~\eqref{BilinearTrace}:
\begin{enumerate}[1)]
\item  
\eq{
\varliminf\limits_{N\to\infty}\avsum\limits_{w\in \mathbb{H}^N}\Big|\sum\limits_{j=0}^N\big(\varphi [v\otimes a]\big)_{w(j)}\Big| = \infty;
}
\item the uniform probability measure on~$\mathbb{H}^N$ fulfills the~$(\alpha,1)$-Frostman condition uniformly with respect to~$N$. 
\end{enumerate}
The second statement follows from the fact that the said measure is represented via the correspondence~\eqref{MeasureGeneratesAMartingale} by the martingale~$\{\prod_{j=0}^n(1+h_j)\}_n$, which is the scalar version of~$F$, and the bound
\eq{
\Big\|\prod\limits_{j=0}^n(1+h_j)\Big\|_{L_\infty} \leq m^{\alpha n}.
}

To prove the first statement, we consider the random variables~$\xi_j$:
\eq{
\xi_j(w) = (\varphi[v\otimes a])_{w(j)},\qquad w\in \mathbb{H}_N,
}
on the probability space~$\mathbb{H}_N$ equipped with the uniform probability measure. Note that these random variables are identically distributed and independent. We need to prove that~$\E |\sum_{j=1}^N\xi_j|\to \infty$ as~$N\to \infty$ (note that the mathematical expectation is computed w.r.t the uniform measure on~$\mathbb{H}_N$). In the case~$\E\xi_j\ne 0$, this follows from the triangle inequality. In the remaining case~$\E\xi_j = 0$, we use the Central Limit Theorem
\eq{
N^{-\frac12}\sum\limits_{j=0}^N\xi_j \rightarrow \mathcal{N}(0,\sigma) \qquad \text{in distribution},
}
here~$\sigma$ is the variance of~$\xi_j$. Thus,~$\E|\sum_{j=0}^N\xi_j| \sim c\sqrt{N}\to\infty$.
\end{proof}

We will need yet another condition imposed on~$W$. The necessity of this condition is doubtable. 
\begin{Def}\label{Non-locality}
Let~$\alpha\in (0,1)$ be fixed. We say that a geometric space~$W \subset V\otimes \R^\ell$ is non-local \textup(of order~$\alpha$\textup) if for any extremal vector~$v\otimes a\in \R^\ell$ supported on~$H$\textup, any~$w\in W$\textup, and any~$b\in \R^\ell$ the identity
\eq{
b+w_j = 0,\qquad \forall j\notin H,
}
implies that~$w$ is proportional to~$v\otimes a$. 
\end{Def}
In particular, the non-locality assumption implies that for any~$H \subset [1\twodots m]$ all~$a \in \R^\ell\setminus \{0\}$ such that there exists an extremal vector~$v\otimes a$ supported on~$H$, are proportional; this somehow resembles the strong cancellation condition in~\cite{GRvS2019} (the similarity is emphasized by the Fourier-side sufficient condition of non-locality in Lemma~\ref{NonLocalCoset} below).
We are ready to formulate our main result --- the endpoint trace theorem.
\begin{Th}\label{Main}
Let~$W$ be a geometric non-local space and let~$\varphi$ be a canceling operator. Then\textup,~\eqref{BilinearTrace} holds true at the endpoint~$\alpha = -\kappa'(1)/\log m$.
\end{Th}
This theorem has a Bellman function behind it.
\begin{Th}\label{Bellster}
Let~$W$ be a geometric non-local space and let~$\varphi$ be a canceling operator. There exists a collection of constants~$1\ll C_1 \ll C_2\ll C_3$ such that the function
\eq{\label{BellsterFormula}
\Phi(x,y,z,t,s) = |y|t + C_1|x|t + C_2|x|\sqrt{st} + C_3(z-|x|)s
}
is a supersolution.
\end{Th}
\begin{Rem}
Mysteriosly\textup, both terms~$|x|t$ and~$|x|\sqrt{st}$ are required.
\end{Rem}

\section{Proof of Theorem~\ref{Bellster}}\label{S5}
In fact, the proof consists of an accurate computation and estimate of~$\NV[\Phi]$, which is done in several steps. 

\paragraph{Estimates for~$|x|\sqrt{st} + M(z-|x|)s$.}
One may verify that this function is a supersolution, provided~$M$ is sufficiently large, using the reasonings presented in the proof of Theorem~\ref{SubcriticalTheorem} (see~\eqref{Discr}). Now we need to prove that this function is a supersolution with certain excess. We derive from~\eqref{MainDiscr} and~\eqref{Holder} that
\eq{
\NV\Big[|x|\sqrt{st} + M(z-|x|)s\Big] \geq \sqrt{st}\Big(|x| - m^{-\alpha/2}\Big(\avsum\limits_{j=1}^{m} |x_j|^2\Big)^\frac12 + M\Big(\avsum\limits_{j=1}^{m}|x_j| - |x|\Big)\Big)
}
(recall~$s\geq t$). By homogeneity with respect to the variables~$x$ and~$z$, we may assume~$|x|=1$. We may also assume~$\avsum_{j=1}^{m}|x_j|\leq (1+\eps)|x|$ for sufficiently small~$\eps$, the other case will be covered by the flat/convex argument and the choice of large~$M$ (or large~$C_3$ for the initial function~$\Phi$). Thus, by Lemma~\ref{FlatRankOne}, we need to study the behavior of the function
\eq{
\Big(|x| - m^{-\alpha/2}\Big(\avsum\limits_{j=1}^{m} |x_j|^2\Big)^\frac12\Big) + M\Big(\avsum\limits_{j=1}^{m}|x_j| - |x|\Big)
}
when~$\vec{x}$ lies in a neighborhood of some~$v\otimes a \in W$ and~$x$ lies in a neighborhood of~$a$. Using the homogeneity with respect to the variables~$x$ and~$z$, we may assume~$|a|=1$ and~$\scalprod{x}{a} = 1$, which does not change the things much since all our functions are homogeneous of order one with respect to~$x$ and~$z$. Of course,~$|x|$ is not necessarily equal to~$1$ anymore. Note that if~$v\otimes a$ is not an extremal vector, then
\eq{\label{L2NonExtremal}
\Big(|x| - m^{-\alpha/2}\Big(\avsum\limits_{j=1}^{m} |x_j|^2\Big)^\frac12\Big) > 0
}
provided~$\vec{x}$ is sufficiently close to~$v\otimes a$ (since this function is positive at~$x=v\otimes a$, see Remark~\ref{ExtremalityRemark}). The case where~$v\otimes a$ is an extremal vector is covered by the following lemma.
\begin{Le}\label{LastTwoSummands}
Assume~$v\otimes a\in W$ is an extremal vector of the geometric non-local space~$W$. If~$\scalprod{x}{a}=1$\textup,~$\vec{x}\in W$ is sufficiently close to~$v\otimes a$\textup, and~$x$ is sufficiently close to~$a$\textup, then
\eq{
\NV\Big[|x|\sqrt{st} + M(z-|x|)s\Big] \gtrsim \Big(\sum\limits_{j=1}^m|x_j - (1+v_j)a|\Big)\sqrt{st},
}
provided~$M$ is sufficiently large\textup; here~$v=(v_1,v_2,\ldots,v_m)$.
\end{Le}
Note that
\eq{
\sum\limits_{j=1}^m|x_j - (1+v_j)a|=\sum\limits_{j\in H}^m|x_j- m^\alpha a| + \sum\limits_{j\notin H}^m|x_j|,
}
where~$H$ is the support of the extremal vector~$v\otimes a$.
\begin{proof}
By the reasonings above, it suffices to prove that
\eq{\label{WithoutST}
\Big(|x| - m^{-\alpha/2}\Big(\avsum\limits_{j=1}^{m} |x_j|^2\Big)^\frac12\Big) + M\Big(\avsum\limits_{j=1}^{m}|x_j| - |x|\Big) \gtrsim \Big(\sum\limits_{j=1}^m|x_j - (1+v_j)a|\Big),
}
given the assumptions of the lemma. Let~$\pi_x\colon \R^\ell \to \R^\ell$ be the orthogonal projection onto the line spanned by~$x$. The notation~$\chi_{\R_-}[\pi_x[y]]$ means the negative part of the projection: it equals~$|\pi_x[y]|$ if~$\scalprod{x}{y} \leq 0$ and zero otherwise. Let~$H$ be the support of~$v\otimes a$. We expand the second term on the left hand side of~\eqref{WithoutST} (see the proof of Lemma~$2.1$ in~\cite{ASW2018} for similar estimates):
\mlt{
\avsum\limits_{j=1}^{m}|x_j| - |x| = \frac{1}{m}\bigg(\sum\limits_{j=1}^m(|x_j| - |\pi_{x}[x_j]|)+\sum\limits_{j=1}^m(|\pi_{x}[x_j]|-|x|)\bigg)\\ \gtrsim  \sum\limits_{j\notin H} (|x_j| - |\pi_x[x_j]|) + \sum\limits_{j\notin H}\chi_{\R_-} [\pi_x[x_j]]. 
}
Let
\eq{\label{Dformula}
D = \Big(\sum\limits_{j=1}^m|x_j - (1+v_j)a|\Big),
}
which is equivalent to the distance between~$\X$ and~$({\bf 1}_{[1..m]} + v)\otimes a$. Then,
\alg{
|x_j|^2 = |m^{\alpha }a|^2 + 2\scalprod{x_j - m^{\alpha} a}{m^{\alpha} a} + O(D^2),\quad &j\in H;\\
|x_j|^2 = O(D^2),\quad &j\notin H;\\
|x|^2 = |a|^2 + 2\scalprod{x-a}{a}+O(D^2) = 1+O(D^2),\quad &
}
by our normalization~$|a|^2 = \scalprod{x}{a}=1$.
Therefore,
\mlt{\label{eq612}
|x| - m^{-\alpha/2}\Big(\avsum\limits_{j=1}^{m} |x_j|^2\Big)^\frac12\\ = 1+O(D^2) - m^{-\frac{\alpha}{2}}\Big(m^{-1}\cdot m^{1-\alpha}\cdot m^{2\alpha} + \frac{2}{m}\sum\limits_{j\in H} \scalprod{x_j - m^{\alpha} a}{m^{\alpha}a} + O(D^2)\Big)^{\frac12} \\= 
1+O(D^2) -\Big(1 + \frac{2}{m^{1+\alpha}}\sum\limits_{j\in H} \scalprod{x_j - m^{\alpha} a}{m^{\alpha}a} + O(D^2)\Big)^{\frac12}\\ = - m^{-1-\alpha}\sum\limits_{j\in H}\scalprod{x_j-m^{\alpha}a}{m^{\alpha} a} + O(D^2).
}
We use the identity~$\scalprod{\sum_{j\in H}x_j}{a} = m-\scalprod{\sum_{j\notin H}x_j}{a}$ (which follows from~$\scalprod{x}{a} = |a| = 1$) to rewrite the expression obtained in~\eqref{eq612} as
\eq{
- m^{-1-\alpha}\sum\limits_{j\in H}\scalprod{x_j-m^{\alpha}a}{m^{\alpha} a} + O(D^2) = \frac{1}{m}\scalprod{\sum\limits_{j\notin H} x_j}{a}+O(D^2).
}
Thus, 
\mlt{\label{eq614}
\Big(|x| - m^{-\alpha/2}\Big(\avsum\limits_{j=1}^{m} |x_j|^2\Big)^\frac12\Big) + M\Big(\avsum\limits_{j=1}^{m}|x_j| - |x|\Big)\\  \gtrsim \frac{1}{m}\scalprod{\sum\limits_{j\notin H} x_j}{a} +  M\bigg(\sum\limits_{j\notin H} (|x_j| - |\pi_x[x_j]|) + \sum\limits_{j\notin H}\chi_{\R_-} [\pi_x[x_j]]\bigg) + O(D^2).
}
{\bf Claim:} the assumption that~$W$ is a non-local space (together with the normalization~$\scalprod{a}{x} = |a|^2$) leads to the estimate
\eq{\label{NonLocalCons}
\sum\limits_{j\in H}|x_j - (1+v_j)a| = \sum\limits_{j\in H}|x_j - m^{\alpha}a| \lesssim \sum\limits_{j\notin H} |x_j|.
}
To prove the claim, consider the linear subspace~$\mathfrak{L}$ of~$\R^m\otimes \R^\ell$ formed by the strings~$(\xi_1,\xi_2,\ldots,\xi_m)$, where~$\xi_j\in \R^\ell$ such that~$(\xi_1-\xi,\xi_2-\xi,\ldots,\xi_m-\xi)\in W$ and~$\scalprod{\xi}{a}=0$; here, as usual,~$\xi = \avsum_{j=1}^m\xi_j$. Note that the string formed of vectors~$x_j - m^\alpha a$ when~$j\in H$ and~$x_j$ when~$j\notin H$, belongs to this space. Therefore, it suffices to prove the inequality
\eq{\label{615}
\|\Xi\| \lesssim\sum\limits_{j\notin H}|\xi_j|,\qquad \Xi = (\xi_1,\xi_2,\ldots,\xi_m)\in \mathfrak{L}. 
}
Note that the function on the right hand side of this inequality is one-homogeneous and even. It suffices to prove this function is positive outside the origin. In such a case, it is a norm on~$\mathfrak{L}$, and the desired inequality~\eqref{615} follows since any two norms on a finite dimensional space are equivalent. To prove 
\eq{
\sum\limits_{j\notin H}|\xi_j| > 0,\qquad \Xi \in \mathfrak{L}\setminus \{0\},
}
assume the contrary: let~$\xi_j=0$ when~$j\notin H$ for some non-zero~$\Xi \in \mathfrak{L}$. Since~$W$ is a non-local space (see Definition~\ref{Non-locality}), this means~$(\xi_1-\xi,\xi_2-\xi,\ldots,\xi_m-\xi)$ is proportional to~$v\otimes a$. In particular, since some of the~$\xi_j$ equals zero,~$\xi$ is parallel to~$a$, which, by our assumption~$\scalprod{\xi}{a} = 0$, means~$\xi = 0$. Therefore,~$\Xi = 0$, a contradiction. Thus, we have proved the {\bf claim}.
\begin{figure}
	\vspace{-50pt}
    \centerline{\includegraphics[width=0.6\textwidth]{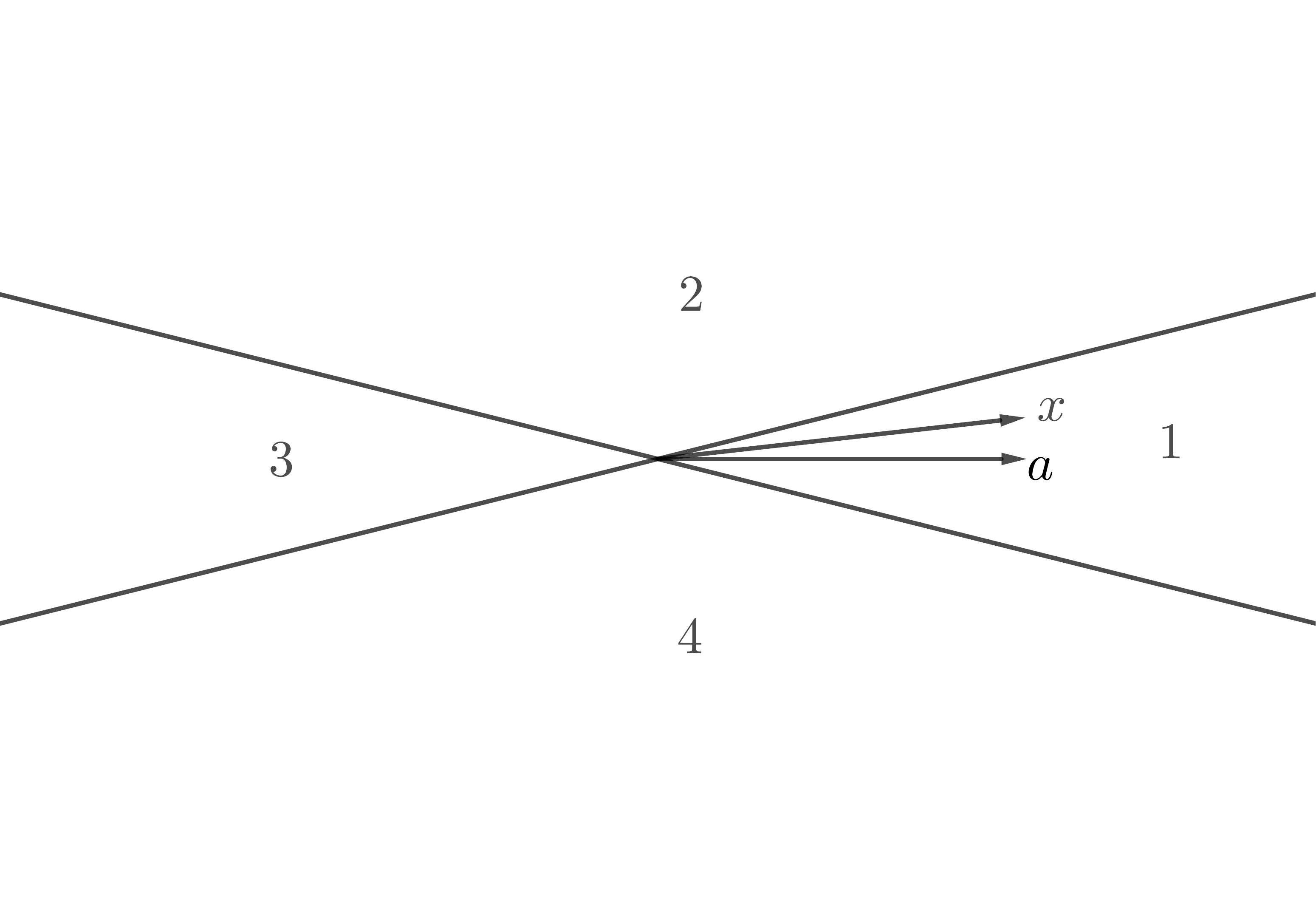}}
    \caption{The plane spanned by~$a$ and~$x$.}
    \label{XA}
\end{figure}

In particular,~$D\lesssim \sum_{j\notin H}|x_j|$. Therefore, by~\eqref{eq614}, the estimate
\eq{\label{Wanted}
\sum\limits_{j\notin H} |x_j| \lesssim \frac{1}{m}\scalprod{\sum\limits_{j\notin H} x_j}{a} +  M\bigg(\sum\limits_{j\notin H} (|x_j| - |\pi_x[x_j]|) + \sum\limits_{j\notin H}\chi_{\R_-} [\pi_x[x_j]]\bigg)
}
yields~\eqref{WithoutST}. Assume~$M > 1$, then, without loss of generality, we may replace the~$x_j$ with their projections onto the linear plane spanned by~$x$ and~$a$ (such a procedure will reduce the values~$|x_j|$ in the formula above, while all the other quantities will remain the same). We will argue geometrically using Fig.~\ref{XA}. We split the plane spanned by~$x$ and~$y$ by the lines~$\scalprod{y}{a} = \pm\frac{1}{3}|y|$ into four regions as it is drawn on Fig.~\ref{XA}.

If~$x_j$ belongs to the second or the fourth domains, then~$|x_j| \lesssim |x_j| - |\pi_x[x_j]|$ (recall that~$x$ is close to~$a$ and, thus, lies strictly inside the first domain). If~$x_j$ lies inside the first domain, then~$|x_j| \lesssim \scalprod{x_j}{a}$. Finally, in the third domain,
\eq{
|x_j| \lesssim \chi_{\R_-} [\pi_x[x_j]],
}
and thus, the choice of large~$M$ implies
\eq{
|x_j| +|\scalprod{x_j}{a}| \leq M \chi_{\R_-} [\pi_x[x_j]] 
} 
in this case, which proves~\eqref{Wanted} (in fact, we have proved the individual estimate
\eq{
|x_j| \lesssim \scalprod{ x_j}{a} +  M\bigg((|x_j| - |\pi_x[x_j]|) + \chi_{\R_-} [\pi_x[x_j]]\bigg)
}
provided~$M$ is sufficiently large).
\end{proof}

\paragraph{Estimates for~$|x|t+L(|x|\sqrt{st} + M(z-|x|)s)$.}
\begin{St}
For sufficiently large~$L$ and~$M$ the function
\eq{
|x|t+L(|x|\sqrt{st} + M(z-|x|)s)
}
is a supersolution.
\end{St}
\begin{proof}
We estimate
\mlt{
\NV[|x|t+L(|x|\sqrt{st} + M(z-|x|)s)] \\\geq \sum\limits_{j=1}^m t_j (|x| - m^{-\alpha} |x_j|)+ Lt\Big(|x| - m^{-\alpha/2}\Big(\avsum\limits_{j=1}^{m} |x_j|^2\Big)^\frac12 + M\Big(\avsum\limits_{j=1}^{m}|x_j| - |x|\Big)\Big);
}
here we have used that~$s_j \geq t_j$ for all~$j$ (and also~\eqref{LastTwoSummands}). Similar to the proof of Theorem~\ref{SubcriticalTheorem}, by Lemma~\ref{FlatRankOne}, it suffices to consider the case where~$\vec{x}$ is close to some~$v\otimes a$ and~$x$ is close to~$a$ (the other cases are covered by a combination of a compactness argument and the flat/convex argument). What is more, we may assume~$v\otimes a$ is an extremal vector (say, by~\eqref{L2NonExtremal}). Let~$H$ be the support of~$v$. We normalize~$x$ in such a way that~$\scalprod{x}{a} = 1$ and use Lemma~\ref{LastTwoSummands} (or~\eqref{WithoutST}):
\eq{
\NV[|x|t+L(|x|\sqrt{st} + M(z-|x|)s)] \geq \sum\limits_{j=1}^m t_j\Big(|x| - m^{-\alpha}|x_j|\Big)+Lt\sum\limits_{i=1}^m|x_i - (1+v_i)a|.
}
Note that for~$j\notin H$ the expression in the parenthesis is always non-negative since~$x_j$ is close~$0$ in this case, while~$x$ is close to~$a$. Fix some~$j\in H$. 
The function~$\X \mapsto |x|-m^{-\alpha}|x_j|$ is Lipschitz on~$\R^{m\ell}$ and vanishes at~$({\bf 1}_{[1..m]}+v)\otimes a$, which means its absolute value is bounded by
\eq{
L\sum\limits_{i=1}^m|x_i - (1+v_i)a|,
}
provided~$L$ is sufficiently large (because the latter expression is the distance between~$\X$ and~$({\bf 1}_{[1..m]}+v)\otimes a$). 

Thus,
\alg{
|x| - m^{-\alpha}|x_j| \geq 0,\qquad & j\notin H;\\
\big||x| - m^{-\alpha}|x_j|\big| \lesssim \sum\limits_{i=1}^m|x_i - (1+v_i)a|,\qquad & j \in H,
}
and the proof is completed by choosing large~$L$.
\end{proof}

\paragraph{End of proof of Theorem~\ref{Bellster}.} We rewrite the discrepancy of the first summand:
\eq{\label{discroftr}
\NV[|y|t] = \sum\limits_{j=1}^m\Big(|y| - |y+m^{-\alpha}(\varphi[\vec{x}])_j|\Big)t_j.
}
We will treat this expression as a function of~$\X$ and~$\mathcal{T} = (t_1,t_2,\ldots,t_m)$. Without loss of generality, we may assume~$|x|=1$ and~$t=1$. The estimate
\eq{
\big||y| - |y+m^{-\alpha}(\varphi[\vec{x}])_j|\big| \lesssim |\X|
}
permits the application of the flat/convex argument, provided we have good estimates in neighborhoods of flat configurations, because if the configuration~$\X$ satisfies~\eqref{Flatness}, then the discrepancy of the fourth term in~\eqref{BellsterFormula} majorizes the discrepancies of the other three terms. We also note that the set of all~$\eps$-flat configurations is compact, provided we require the normalization~$|x|=1$. In particular, given any flat configuration~$(({\bf 1}_{[1..m]}+v)\otimes a,\mathcal{T}_0)$, it suffices to find its neighborhood and a large constant~$K$ such that
\eq{
\NV\Big[|y|t + K\big(|x|t + L(|x|\sqrt{st} + M(z-|x|))\big)\Big] \geq 0
}
when~$(\X,\mathcal{T})$ lies inside the said neighborhood and~$y$,~$\Z=(z_1,z_2,\ldots,z_m)$, and~$\SSS = (s_1,s_2,\ldots,s_m)$ are arbitrary. If~$v\otimes a$ is not extremal, then the term
\eq{
\NV[|x|t + L(|x|\sqrt{st} + M(z-|x|))]
}
is bounded away from zero (uniformly with respect to~$\Z$ and~$\SSS$) in a sufficiently small neighborhood of~$(({\bf 1}_{[1..m]}+v)\otimes a,\mathcal{T}_0)$ (say, by~\eqref{L2NonExtremal}), and the problem reduces to the choice of sufficiently large~$K$. Assume~$v\otimes a$ is extremal. As usual, we renormalize~$x$ to~$\scalprod{x}{a}=1$. We estimate
\eq{
\NV[|y|t] \geq -\sum\limits_{j\notin H}|\X|t_j -\sum\limits_{j\in H}\Big(\sum\limits_{i=1}^m|x_i - (1+v_i)a|\Big)t_j;
}
this bound comes from~\eqref{discroftr} and the cancellation condition imposed on~$\varphi$ (for any~$j\in H$ we have~$(\varphi(v\otimes a))_j = 0$; recall also that~$D$ given in~\eqref{Dformula} is the distance between~$\X$ and~$({\bf 1}_{[1..m]} + v)\otimes a$). The second sum is compensated by the discrepancy of~$|x|\sqrt{st} + M (z-|x|)s$, by Lemma~\ref{LastTwoSummands}. As for the first sum, we have
\eq{
\sum\limits_{j\notin H}|\X|t_j\lesssim \sum\limits_{j\notin H}(|x| -m^{-\alpha}|x_j|)t_j,
}
provided~$x$ is sufficiently close to~$v\otimes a$. Thus, this part of the discrepancy may be compensated by~$\NV[|x|t]$ by choosing sufficiently large~$C_1$.\qed

\section{Two examples}\label{S6}
Let now~$m = \mu d$, where~$\mu$ and~$d$ are natural numbers. We equip the set~$[1\twodots m]$ with the group structure of~$(\ZZ_\mu)^d$; this also identifies it with the cube~$[0\twodots \mu-1]^d$ in a natural way. Let
\eq{
e_{j} = (\underbrace{0,0,\ldots,0}_{j-1\, \text{times}},1,0,\ldots,0) \in [0\twodots\mu-1]^d,\quad j=1,2,\ldots,d,
}
be the generators of~$(\ZZ_\mu)^d$. It will be convenient to consider complex-valued functions and martingales, i.e. we set~$\ell = 2l$ and identify~$\C^l$ and~$\R^\ell$ in a natural way.  In our two main examples~$l = d$. The first example will model the space~$\BV(\R^d)$,
\eq{\label{Wg}
\Wg = \Set{g\in V\otimes \C^d}{\exists f\in \R^m \quad \text{s. t.}\quad \forall x\in [0\twodots\mu-1]^d, j \in [1\twodots d],\quad g_j(x) = f(x+e_j) - f(x)}.
}
The second will model the space of solenoidal charges,
\eq{\label{Wd}
\Wd = \Set{g\in V\otimes \C^d}{\forall x\in [0\twodots \mu-1]^d\qquad \sum\limits_{j=1}^d(g_j(x+e_j) - g_j(x)) = 0}.
}
Before we pass to the details, we survey simple facts about translation invariant spaces. 
\begin{Def}
We say that a linear subspace~$W\subset V\otimes \C^l$ is translation invariant if for any~$f\in W$ and~$y \in (\ZZ_\mu)^d$\textup, the function~$[0\twodots \mu-1]^d\ni x\mapsto f(x+y)$ also belongs to~$W$\textup, the addition sign means addition in~$(\ZZ_\mu)^d$. 
\end{Def}
Note that the spaces~$\Wg$ and~$\Wd$ are translation invariant. Let us define the Fourier transform of a function~$f\colon (\ZZ_\mu)^d\to \C^l$ by the formula
\eq{
\hat{f}(\gamma) = \sum\limits_{x\in (\ZZ_\mu)^d}e^{-2\pi i \frac{\gamma \cdot x}{\mu}}f(x),\qquad \gamma \in (\ZZ_\mu)^d,
}
where~$\gamma\cdot x$ is short for~$\sum_{j=1}^d\gamma_jx_j\in \ZZ_\mu$, and we interpret this expression as a residue modulo~$\mu$ (a natural number in~$[0\twodots\mu-1]$). The inverse Fourier transform is then given by the formula
\eq{
\check{g}(x) = m^{-1}\sum\limits_{\gamma\in (\ZZ_\mu)^d}e^{2\pi i \frac{\gamma \cdot x}{\mu}}g(\gamma),\qquad x \in (\ZZ_\mu)^d.
}
We refer the reader to the book~\cite{Rudin1962} for introduction to abstract Fourier analysis (though our case of finite groups is elementary, sometimes the abstract intuition seems useful).
The following lemma is simple, so, we leave it without proof.
\begin{Le}
The space~$W \subset V\otimes \C^l$ is translation invariant if and only if there exist linear spaces~$\Omega(\gamma) \subset \C^l$\textup,~$\gamma \in (\ZZ_\mu)^d\setminus \{0\}$\textup, such that
\eq{
W = \Set{f\in V\otimes \C^l}{\forall \gamma \in (\ZZ_\mu)^d\setminus \{0\}\qquad \hat{f}(\gamma) \in \Omega(\gamma)}.
}
\end{Le}
\begin{Rem}
Since any element of~$W$ has mean zero\textup,~$\hat{f}(0) = 0$ for any~$f\in W$. We may set~$\Omega(0) = \{0\}$.
\end{Rem}
In the case~$W = \Wg$, the corresponding spaces~$\Omega(\gamma)$ are complex lines defined by the formula
\eq{\label{OmegaGrad}
\Og(\gamma) = \big(e^{2\pi i\frac{\gamma_1}{\mu}}-1, e^{2\pi i\frac{\gamma_2}{\mu}}-1,\ldots, e^{2\pi i\frac{\gamma_d}{\mu}}-1\big)\cdot\C,\qquad \gamma = (\gamma_1,\gamma_2,\ldots,\gamma_d)\in (\ZZ_{\mu})^d\setminus\{0\}.
}
In the case~$W = \Wd$ the spaces~$\Omega(\gamma)$ are complex hyperplanes
\eq{
\Od(\gamma) = \Set{\zeta \in\C^l}{\sum\limits_{j=1}^d\zeta_j(e^{2\pi i \frac{\gamma_j}{\mu}}-1) = 0}, \qquad \gamma = (\gamma_1,\gamma_2,\ldots,\gamma_d)\in (\ZZ_{\mu})^d\setminus\{0\}.
}

If~$W$ is a translation invariant space, the sets~$\Omega^{-1}(a)$, where~$a\in \C^l$, play an important role:
\eq{
\Omega^{-1}(a) = \Set{\gamma \in (\ZZ_\mu)^d}{a\in \Omega(\gamma)},
}
since these sets contain the spectra of rank-one vectors. By spectrum of a function we mean the support of its Fourier transform.
\begin{Le}\label{FourierLemma}
Let~$k \in [0\twodots d]$. Assume that~$W$ is a translation invariant subspace of~$V\otimes \C^l$ such that
\eq{\label{AntiSymmetry}
\#\big(\Omega^{-1}(a)\cap(-\Omega^{-1}(a))\big) \leq \mu^{k}-1
} 
for any~$a\in \C^l\setminus \{0\}$. Assume~\eqref{AntiSymmetry} turns into equality for at least one vector~$a$ and that for this special~$a$ we have~$\Omega^{-1}(a) = \Gamma\setminus \{0\}$\textup, where~$\Gamma$ is a subgroup of index~$d-k$ in~$(\ZZ_\mu)^d$. Then\textup,~$W$ is a geometric space of order~$k/d$.
\end{Le}
\begin{proof}
Let~$v\otimes a \in V\otimes \C^l$ be a rank-one vector. By definition,
\eq{
\hat{v}(\gamma)a\in \Omega(\gamma),\qquad \gamma \in (\ZZ_\mu)^d.
}
Here we treat~$v$ as a real valued function on~$[0\twodots\mu-1]^d$. Therefore,~$a\in \Omega(\gamma)$ whenever~$\hat{v}(\gamma) \ne 0$. Since~$\hat{v}(-\gamma) = \overline{\hat{v}(\gamma)}$, we have
\eq{
\supp \hat{v} \subset \Omega^{-1}(a)\cap(-\Omega^{-1}(a)).
}
Assume now that~$v(x) \geq -1$ for any~$x \in [0\twodots\mu-1]^d$ and let~$w(x) = 1+v(x)$. Then,~$\hat{w}$ is supported in
\eq{
\big(\Omega^{-1}(a)\cap(-\Omega^{-1}(a))\big)\cup \{0\}
}
and, since~$\sum_x w(x) = m$ and the numbers~$w(x)$ are non-negative,~$|\hat{w}(\gamma)| \leq m$ for any~$\gamma$. From this information and the bound~\eqref{AntiSymmetry}, we conclude that
\eq{\label{L_1Estimates}
|w(x)| = m^{-1}\Big|\sum\limits_{\gamma\in (\ZZ_\mu)^d}e^{2\pi i \frac{\gamma \cdot x}{\mu}}\hat{w}(\gamma)\Big| \leq \mu^{k}, \qquad x\in [0\twodots \mu-1]^d.
}
Therefore, by~\eqref{Kappa0},~$\kappa_v(0) \leq k \log \mu$. By convexity of the function~$\kappa_v$ and the identity~$\kappa_v(1) = 0$,
\eq{
\kappa_v(\theta) \leq k(\log \mu) (1-\theta).
}
On the other hand, if the set
\eq{
\big(\Omega^{-1}(a)\cap(-\Omega^{-1}(a))\big)\cup \{0\}
}
is a subgroup~$\Gamma$, then we set~$w = |\Gamma|\chi_{\Gamma^\perp}$ and see that~$v\otimes a \in W$ for the corresponding~$v$. For such a choice of~$v$ and~$w$, the inequalities~\eqref{L_1Estimates} turn into equalities when~$x\in \Gamma^\perp$, which means
\eq{
\kappa_v(\theta) = k\log \mu (1-\theta) = \frac{k}{d}\log m (1 - \theta),
}
and~$\Gamma^\perp$ is the support of an extremal vector. Thus, the function~$\kappa$ is linear and~$W$ is a geometric space.
\end{proof}
\begin{Rem}\label{SupportsOfExtremalVectors}
The proof above says that extremal vectors correspond to those~$a\in \C^l\setminus \{0\}$\textup, for which the set
\eq{
\big(\Omega^{-1}(a)\cap(-\Omega^{-1}(a))\big)\cup \{0\}
}
is a subgroup~$\Gamma$ of index~$d-k$. If~$v\otimes a$ is an extremal vector\textup, then its support is a coset of~$\Gamma^\perp$.
\end{Rem}
\begin{Def}
A subset~$\ell_D\subset (\ZZ_\mu)^d$ is called a combinatorial line generated by a set~$D \subset [0\twodots\mu-1]$ provided
\eq{
\ell_D = \Set{x\in (\ZZ_\mu)^d}{\exists b \in \ZZ_\mu\quad\quad \text{ such that } x_j = 0,  j \notin D,\quad x_j = b, j \in D}.
}
\end{Def}
\begin{Le}\label{WgLemma}
Let~$W = \Wg$. Then\textup, for any~$a\in \C^d\setminus \{0\}$\textup, either
\eq{
(\Og^{-1}(a)\cap (-\Og^{-1}(a))) \cup \{0\} = \ell_D,
}
for some~$D \subset [1\twodots d]$\textup, or~$\Og^{-1}(a)\cap (-\Og^{-1}(a)) = \varnothing$.
\end{Le}
\begin{proof}
Let~$a\in \C^l \setminus \{0\}$. Assume some points~$\gamma^1,\gamma^2,\ldots,\gamma^s$ belong to~$\Omega^{-1}_\nabla(a)$. By~\eqref{OmegaGrad}, this means there exist non-zero complex numbers~$\lambda_1,\lambda_2,\ldots,\lambda_s$ such that
\eq{
a = \lambda_k \big(e^{2\pi i\frac{\gamma_1^k}{\mu}}-1, e^{2\pi i\frac{\gamma_2^k}{\mu}}-1,\ldots, e^{2\pi i\frac{\gamma_d^k}{\mu}}-1\big),\qquad k= 1,2,\ldots,s.
}
Define
\eq{\label{DDefin}
D = \set{j\in[1\twodots d]}{a_j \ne 0}.
}
Note that~$\gamma^k_j = 0$ for any~$k\in[1\twodots s]$ and~$j\notin D$. If there are at least two distinct non-zero values~$a_{j_1}$ and~$a_{j_2}$, then the ratio
\eq{
1\ne \frac{a_{j_1}}{a_{j_2}} = \frac{e^{2\pi i\frac{\gamma_{j_1}^k}{\mu}}-1}{e^{2\pi i\frac{\gamma_{j_2}^k}{\mu}}-1}
}
completely defines the elements~$\gamma_{j_1}^k$ and~$\gamma_{j_2}^k$ (i.e. these values do not depend on~$k$) since the equation
\eq{
e^{ix} - 1 =\alpha (e^{iy}-1),\qquad x,y\in (0,2\pi),
}
has at most one solution when~$\alpha \ne 1$.  Thus, either all non-zero~$a_j$ are equal, or~$s=1$. In the latter case~$a\notin \Omega_\nabla(-\gamma)$. In the former case,~$\Omega_\nabla^{-1}(a)\cap(-\Omega_\nabla^{-1}(a)) = \ell_D$ with~$D$ given in~\eqref{DDefin}.
\end{proof}
\begin{Cor}
The space~$\Wg$ is a geometric space of order~$1/d$.
\end{Cor}
\begin{proof}
We note that each combinatorial line is a subgroup of index~$d-1$ and use Lemma~\ref{FourierLemma}.
\end{proof}
\begin{Rem}\label{GradientExtremalVectors}
The extremal vectors of~$\Wg$ are easy to describe\textup: by Remark~\textup{\ref{SupportsOfExtremalVectors}}, each such vector is defined by its support\textup, which is a coset of the subgroup~$\ell_D^\perp$\textup,~$D\subset [1\twodots d]$\textup:
\eq{\label{WgExtremalSets}
\ell_D^\perp = \Set{x\in [0\twodots\mu-1]^d}{\sum\limits_{j\in D}x_j = 0}; 
}
the corresponding vector~$a$ is parallel to~$\sum_{j\in D}e_j$.
\end{Rem}
It is harder to describe the sets~$\Od^{-1}(a)$ in general. We will only prove that apart from several exclusions, these sets are small provided~$\mu$ is sufficiently large. 
\begin{Le}\label{DivFibers}
Let~$a\in \C^l\setminus \{0\}$. Either~$\#\Od^{-1}(a)\lesssim \mu^{d-2}$\textup, or 
\eq{\label{Ex1}
\exists j\in [1\twodots d]\text{ such that } \qquad \Od^{-1}(a) = \Set{\gamma \in (\ZZ_\mu)^d}{\gamma_j=0}, \text{or}
}
\eq{\label{Ex2}
\exists j_1\ne j_2\in [1\twodots d]\text{ such that} \qquad \Od^{-1}(a) = \Set{\gamma\in (\ZZ_\mu)^d}{\gamma_{j_1} = \gamma_{j_2}}.
}
\end{Le}
\begin{Rem}\label{DivergenceExtremalVectors}
The set~\eqref{Ex1} corresponds to~$a$ being the~$j$-the basic vector in~$\C^d$ \textup(or a multiple of it\textup)\textup, the set~\eqref{Ex2} corresponds to~$a$ being the difference of~$j_1$-th and~$j_2$-th basic vectors \textup(or a mutliple of it\textup). Using Remark~\textup{\ref{SupportsOfExtremalVectors},} we may describe the extremal vectors. In the case~\eqref{Ex1}\textup, each such vector is supported by a coset of the subgroup
\eq{
\Gamma^\perp = \ZZ_\mu e_j;
}
in other words\textup, for any extremal vector of this type there exists~$x\in (\ZZ_\mu)^d$ such that it is supported by~$x + \ZZ_\mu e_j$. In the case~\eqref{Ex2}\textup, the support of an extremal vector is a coset of the subgroup
\eq{
\Gamma^\perp = \Set{x\in (\ZZ_\mu)^d}{x_{j_1}+x_{j_2} = 0}.
}
\end{Rem}
The proof of Lemma~\ref{DivFibers} will take some time. We need an auxiliary lemma. Let~$b\in \C$. Consider the set
\eq{
S_{a,b} = \Set{\zeta \in (\ZZ_\mu)^d}{\sum\limits_{j=1}^da_j(e^{\frac{2\pi i \zeta_j}{\mu}}-1) = b}.
}
\begin{Le}\label{Lemma710}
Let all~$a_j$\textup,~$j=1,2,\ldots,d$\textup, be non-zero. If~$d \geq 3$\textup, then~$\#S_{a,b}\lesssim \mu^{d-2}$. In the case~$d=2$ either~$\#S_{a,b} \leq 2$ or~$a_1/a_2 = e^{2\pi i \frac{\zeta}{\mu}}$ for some~$\zeta \in [0\twodots\mu-1]$\textup,~$b = -a_1-a_2$\textup, and~$\#S_{a,b} = \mu$ in this case.
\end{Le} 
\begin{proof}
Let us study the structure of the sets~$S_{a,b}$ in the case~$d=2$ first.

We rewrite the equation defining~$S_{a,b}$ in the form
\eq{
e^{\frac{2\pi i \zeta_1}{\mu}} = -\frac{a_2}{a_1}e^{\frac{2\pi i \zeta_2}{\mu}} + \frac{a_2 + b + a_1}{a_1}.
}
The left hand and right hand sides define two circles in the complex plane. In general, such two circles intersect by at most two points, which leads to the estimate~$\#S_{a,b}\leq 2$. However, there is one exception: the two circles might coincide; in such a case, we cannot hope for anything better than~$\#S_{a,b}\leq \mu$ (this estimate is definitely true). If the circles coincide, then~$|a_1| = |a_2|$ and~$b = -a_1-a_2$. Since we search for solutions~$\zeta_1,\zeta_2 \in [0\twodots\mu-1]$, we also need~$a_1/a_2 = e^{2\pi i \frac{\zeta}{\mu}}$, where~$\zeta \in [0\twodots\mu-1]$, to have~$\#S_{a,b} = \mu$. We see that for any pair of non-zero numbers~$(a_1,a_2)$ there exists at most one~$b$ such that~$\#S_{a,b} > 2$.

Now let us pass to the case~$d=3$ and prove that in this case~$\#S_{a,b} \leq 3\mu$. We consider the sets
\eq{
S_{a,b,\zeta_1} = \Set{(\zeta_2,\zeta_3)\in (\ZZ_\mu)^2}{a_2(e^{2\pi i\frac{\zeta_{2}}{\mu}}-1) + a_3(e^{2\pi i\frac{\zeta_{3}}{\mu}}-1) = b - a_1(e^{2\pi i\frac{\zeta_{1}}{\mu}}-1)}.
} 
We note that since~$a_1 \ne 0$,~$\#S_{a,b,\zeta_1} \leq 2$ for all but at most one values of~$\zeta_1$ (for this special value, we have~$\#S_{a,b,\zeta_1} \leq \mu$). Summing over all~$\zeta_1$, we get~$\#S_{a,b}\leq 3\mu$.

Using induction and a similar 'slicing' trick, we prove that~$\#S_{a,b}\leq 3 \mu^{d-2}$ for all larger~$d$. 
\end{proof}
\begin{proof}[Proof of Lemma~\ref{DivFibers}]
Now we allow some of the numbers~$a_j$ to be equal zero. If~$d-1$ of them vanish, we have~\eqref{Ex1}. If~$d-2$ of them vanish, let~$a_{j_1}$ and~$a_{j_2}$ be non-zero. Then,
\eq{
\#\Od^{-1}(a)\leq \mu^{d-2}\#S_{(a_{j_1}, a_{j_2}),0}.
}
If~$a_{j_1}\ne -a_{j_{2}}$, then~$\#S_{(a_{j_1}, a_{j_2}),0} \leq 2$ by Lemma~\ref{Lemma710}. If~$a_{j_1} = -a_{j_2}$, then we arrive at~\eqref{Ex2}. Finally, if~$d-k$ of the coordinates of~$a$ vanish,~$k > 2$, then
\eq{
\#\Od^{-1}(a)\lesssim \mu^{d-k} \mu^{k-2} \leq \mu^{d-2}
} 
by Lemma~\ref{Lemma710}.
\end{proof}
\begin{Rem}
We have proved that apart from the exceptions\textup,~$\#\Od^{-1}(a)\leq 3\mu^{d-2}$\textup, which is less than~$\mu^{d-1}$ if~$\mu \geq 4$.
\end{Rem}
\begin{Cor}
For~$\mu\geq 4$\textup, the space~$\Wd$ is geometric of order~$(d-1)/d$.
\end{Cor}
\begin{proof}
We note that the sets~\eqref{Ex1} and~\eqref{Ex2} are subgroups of index~$1$ and apply Lemma~\ref{FourierLemma}.
\end{proof}
Now we pass to verification of the non-locality condition given in Definition~\ref{Non-locality}. 
\begin{Le}\label{NonLocalCoset}
Let~$k \in [0\twodots d]$. Assume that~$W$ is a translation invariant subspace of~$V\otimes \C^l$ such that
\eq{
\#\big(\Omega^{-1}(a)\cap(-\Omega^{-1}(a))\big) \leq \mu^{k}-1
} 
for any~$a\in \C^l\setminus \{0\}$. Assume also that if~$\Omega^{-1}(a) = \Gamma\setminus \{0\}$\textup, where~$\Gamma$ is a subgroup of index~$d-k$ in~$(\ZZ_\mu)^d$\textup, then
\alg{
&\label{ParallelToa}\bigcap\limits_{\gamma \in \Gamma}\Omega(\gamma) = \C a\  \quad\text{and}\\
&\label{parallelto0}\bigcap\limits_{\gamma \in S}\Omega(\gamma) = \{0\}
} 
for any coset~$S\ne \Gamma$ of~$\Gamma$.
Then\textup,~$W$ is a non-local space of order~$k/d$.
\end{Le}
\begin{proof}
By Remark~\ref{SupportsOfExtremalVectors} and translation invariance, it suffices to prove that if the function~$c+f$, where~$c$ is a constant and~$f\in W$, is supported in~$\Gamma^\perp$ for some subgroup~$\Gamma$ of index~$d-k$,~$\Gamma = \Omega^{-1}(a)\cup \{0\}$, then~$c+f$ is proportional to~$\chi_{\Gamma^\perp}\otimes a$. If~$g = c+f$ is supported in~$\Gamma^\perp$, then its Fourier transform is constant on each coset of~$\Gamma$. In particular, for any such coset~$S$ and any~$\zeta \in S$,
\eq{
\hat{g}(\zeta) \in \bigcap_{\gamma \in S}\Omega(\gamma).
}
By~\eqref{parallelto0},~$\hat{g}$ is supported in~$\Gamma$ and, by~\eqref{ParallelToa},~$\hat{g}$ is proportional to~$\chi_\Gamma\otimes a$. Therefore,~$g$ is proportional to~$\chi_{\Gamma^\perp}\otimes a$.
\end{proof}
\begin{Cor}
Let~$\mu \geq 4$. The space~$\Wg$ is non-local of order~$1/d$ and the space~$\Wd$ is non-local of order~$(d-1)/d$.
\end{Cor}
\begin{proof}
We need to verify the assumptions of Lemma~\ref{NonLocalCoset} for the cases~$\Wg$ and~$\Wd$. The first case follows easily from Lemma~\ref{WgLemma}, the second case follows from Lemma~\ref{DivFibers}.
\end{proof}

Let now~$\varphi \colon W \to V$ be an operator that commutes with translations. We may extend it to an operator mapping~$\R^m\otimes \C^l$ to~$V$ and then average over translations. The resulting operator will be translation invariant. In particular, there exists a function~$M\colon (\ZZ_\mu)^d\to \C^l$ such that
\eq{\label{Mgamma}
\mathcal{F}[\varphi[f]](\gamma) = \scalprod{M(\gamma)}{\hat{f}(\gamma)},\qquad \gamma \in (\ZZ_\mu)^d,\ f \in W,
}
the symbol~$\mathcal{F}$ means the Fourier transform. Here we use scalar product in~$\C^l$. Without loss of generality, we may assume~$M(\gamma) \in \Omega(\gamma)$ (in case this relation does not hold for some~$\gamma$, we may redefine~$M(\gamma)$ to be the orthogonal projection of~$M(\gamma)$ onto~$\Omega(\gamma)$; this will not change the formula~\eqref{Mgamma}). In particular,~$M(0) = 0$. Using the inverse Fourier transform, we have
\eq{\label{Convolution}
\varphi[f](x) = \sum\limits_{y\in(\ZZ_\mu)^d}\scalprod{K(x-y)}{f(y)},\qquad x\in (\ZZ_\mu)^d;
}
here~$K = m\check{M}$. Assume now that~$W$ satisfies the assumptions of Lemma~\ref{FourierLemma}. Let~$v\otimes a$ be an extremal vector. Then, by translation invariance and the condition~$\sum_x K(x) = 0$, the cancellation condition from Definition~\ref{CancelingOp} is equivalent to
\eq{
0=\varphi[v\otimes a](0) = \mu^\alpha\sum\limits_{y\in H} \scalprod{K(y)}{a},
}
here we assume~$v\otimes a$ is supported by a subgroup~$H$ (see Remark~\ref{SupportsOfExtremalVectors}). In the case~$W = \Wg$, this reads as follows: 
\eq{\label{CancellationGrad}
\forall \ D\subset [1\twodots d]\qquad \sum\limits_{x\colon\!\! \sum\limits_{j\in D} x_j = 0} \scalprod{K(x)}{e_D} = 0, \qquad e_D=\sum\limits_{j\in D}e_j.
}
\begin{Th}\label{GraadTh}
Let~$W$ be given by~\eqref{Wg}. Let the operator~$\varphi$ be defined by~\eqref{Convolution}. Then\textup, the inequality~\eqref{BilinearTrace} holds true if and only if the operator~$\varphi$ fulfills the cancellation condition~\eqref{CancellationGrad}.
\end{Th}
In the case~$W = \Wd$, the cancellation condition turns into:
\alg{
&\label{CancellationDiv1}\forall \ j \in [1\twodots d]&\qquad &\sum\limits_{y\in \ZZ_\mu} K_j(y e_j) = 0;\\
&\label{CancellationDiv2}\forall \ i\ne j \in [1\twodots d]& \qquad &\sum\limits_{y\in \ZZ_\mu} K_j(y(e_j - e_i)) - K_{i}(y(e_j - e_i)) = 0.
}
\begin{Th}\label{DivTh}
Let~$W$ be given by~\eqref{Wd}. Let the operator~$\varphi$ be defined by~\eqref{Convolution}. Then\textup, the inequality~\eqref{BilinearTrace} holds true if and only if the operator~$\varphi$ fulfills the cancellation conditions~\eqref{CancellationDiv1} and~\eqref{CancellationDiv2}.
\end{Th}
Compare with condition~$(1.10)$ in~\cite{RSS2021} and Conjecture~$1.6$ of the same paper. It says that the inequality
\eq{
\|K*f\|_{L_1(\nu)}\lesssim \|f\|_{L_1(\R^d)}, \qquad \mathrm{div}\, f = 0, \quad K*f = \sum\limits_{j=1}^d K_j*f_j,
} 
holds true for all measures~$\nu$ satisfying the ball growth (or Frostman) condition~$\nu(B_r(x))\lesssim r$, if and only if the kernel~$K\colon \R^d \to \R^d$ that is homogeneous of order~$-1$ satisfies the condition
\eq{
\sum\limits_{j=1}^d (K_j(x) + K_j(-x))x_j = 0
}
for any~$x\in \R^d\setminus \{0\}$.

Note that since the kernel~$K = (K_1,K_2,\ldots,K_d)$ is homogeneous of order~$-1$, the latter condition may be informally interpreted that the (divergent) integral~$\int_{\R} \scalprod{K(tx)}{x}\,dt$ vanishes for any~$x$, which is somehow similar to~\eqref{CancellationDiv1} and~\eqref{CancellationDiv2}. In a similar manner, Theorem~\ref{GraadTh} leads to another natural conjecture.
\begin{Conj}
Assume the kernel~$K\colon \R^d\to \R^d$ is homogeneous of order~$1-d$ and smooth outside the origin. The inequality
\eq{
\|K*\nabla f\|_{L_1(\nu)}\lesssim \|\nabla f\|_{L_1(\R^d)}, \qquad f\in C_0^\infty(\R^d),
}
holds true for all measures~$\nu$ satisfying the condition~$\nu(B_r(x))\lesssim r^{d-1}$ if and only if for any~$e \in S^{d-1}$
\eq{
\int\limits_{e^\perp \cap S^{d-1}} \scalprod{K(\zeta)}{e}\,d\zeta = 0.
} 
The integration is performed with respect to the natural Hausdorff measure on the~$(d-2)$-dimensional sphere~$e^{\perp} \cap S^{d-1}$.
\end{Conj}
\bibliography{mybib}{}
\bibliographystyle{amsplain}

Dmitriy Stolyarov

St. Petersburg State University, 14th line 29B, Vasilyevsky Island, St. Petersburg, Russia.

d.m.stolyarov@spbu.ru.

\end{document}